\documentclass[12pt]{amsart}

\usepackage{amsfonts,amssymb,amsmath, graphicx,color,epic}
\usepackage{amsthm}

\usepackage{pgf, tikz}

\usepackage[textsize=tiny]{todo}

\usepackage{float}
\usepackage[textwidth=15cm, left=3.3cm, top=3cm, bottom=3cm]{geometry}
\usepackage[font=small, margin=1cm]{caption}

\theoremstyle{plain}
\newtheorem{teo}{Theorem}
\newtheorem{prop}{Proposition}
\newtheorem{defin}{Definition}
\newtheorem{lem}{Lemma}
\newtheorem{corol}{Corollary}

\theoremstyle{definition}
\newtheorem{rem}{Remark}

\newcommand{\ds}{\displaystyle}
\newcommand{\cA}{\mathcal{A}}

\newcommand{\cC}{\mathcal{C}}

\newcommand{\cF}{\mathcal{F}}
\newcommand{\cG}{\mathcal{G}}

\newcommand{\cL}{\mathcal{L}}

\newcommand{\cP}{\mathcal{P}}

\newcommand{\cW}{\mathcal{W}}

\newcommand{\bR}{\mathbb{R}}

\newcommand{\bP}{\mathbb{P}}

\newcommand{\bN}{\mathbb{N}}

\newcommand{\bZ}{\mathbb{Z}}

\newcommand{\bA}{\mathsf{A}}

\newcommand{\tX}{\widetilde{X}}
\newcommand{\tY}{\widetilde{Y}}
\newcommand{\tx}{\widetilde{x}}

\newcommand{\tg}{\widetilde{g}}
\newcommand{\tA}{\widetilde{A}}
\newcommand{\tB}{\widetilde{B}}
\newcommand{\tC}{\widetilde{C}}
\newcommand{\tDelta}{\widetilde{\Delta}}

\newcommand{\ovA}{\overline{A}}
\newcommand{\ovB}{\overline{B}}
\newcommand{\ovC}{\overline{C}}
\newcommand{\ovx}{\overline{x}}

\newcommand{\ovX}{\overline{X}}
\newcommand{\ovY}{\overline{Y}}
\newcommand{\ovg}{\overline{g}}

\newcommand{\ovxi}{\overline{\xi}}
\newcommand{\ovomega}{\overline{\omega}}
\newcommand{\ovDelta}{\overline{\Delta}}

\newcommand{\dado}{\,\vert \,}
\newcommand{\Dado}{\bigm\vert}
\newcommand{\DDado}{\Bigm\vert}

\newcommand{\txi}{\tilde{\xi}}
\newcommand{\tvarphi}{\widetilde{\varphi}}
\newcommand{\ovvarphi}{\overline{\varphi}}
\newcommand{\thh}{\tilde{h}}

\newcommand{\var}{\mathrm{var}}

\newcommand{\comeco}{\operatorname{\cC}}
\newcommand{\act}{\operatorname{\cA}}
\newcommand{\diam}{\operatorname{d}}

\newcommand{\palavra}{\operatorname{\Pi}}

\newcommand{\supp}{\operatorname{supp}}

\newcommand{\abs}[1]{\lvert #1 \rvert}
\newcommand{\pordef}{:=}

\newcommand{\grasA}[1]{#1}

\begin{document}
\title[Complete connections with
modified majority rules]{Uniqueness vs. non-uniqueness in 
complete connections with modified majority 
rules}

\begin{abstract}
We take a closer look at a class of chains with complete connections
inspired by the one of Berger,
Hoffman and Sidoravicius \cite{BHS}. Besides giving a sharper description of
the uniqueness and non-uniqueness regimes, we show that if the pure majority
rule used to fix the dependence on the past is replaced with a function that is
Lipschitz at the origin, then uniqueness always holds, even with arbitrarily slow
decaying variation.
\end{abstract}

\author{J. C. A. Dias}
\address{Departamento de Matem\'atica, Universidade Federal de
Ouro Preto, Morro do Cruzeiro, CEP 35400-000, Ouro Preto, Brasil}
\curraddr{Departamento de Matem\'atica, Universidade Federal de
Minas Gerais, Av. Ant\^onio Carlos 6627, C.P. 702 CEP 30123-970, 
Belo Horizonte, Brasil}

\author{S. Friedli}
\address{Departamento de Matem\'atica, Universidade Federal de
Minas Gerais, Av. Ant\^onio Carlos 6627, C.P. 702 CEP 30123-970, 
Belo Horizonte, Brasil}

\thanks{J.C.A.D. was partially supported by CAPES/REUNI}
\maketitle

\section{Introduction}

We consider stationary stochastic processes on $\bZ$, 
\[\dots,Z_{-2},Z_{-1},Z_0,Z_1,Z_2,\dots\]
where each $Z_t$, $t\in \bZ$, is a
symbol taking values in a finite alphabet $\bA$. The processes
we consider are called \grasA{chains with complete
connections} (Doeblin and Fortet \cite{DoeblinFortet}), 
due to a dependence on the past of the following form. 
Assume some measurable map 
$g:\bA\times \bA^{\bN}\to [0,1]$ is given a priori,
called \grasA{g-function}, and that 
for all $t$, all $z_{t}\in \bA$,
\begin{equation}\label{eq_propr_g_process}
P(Z_{t}=z_{t}|Z_{t-1}=z_{t-1},Z_{t-2}=z_{t-2},\dots)
=g(z_{t}|z_{t-1},z_{t-2},\dots)\quad\text{ a.s.}
\end{equation}
A processes 
$Z=(Z_t)_{t\in \bZ}$ satisfying \eqref{eq_propr_g_process} 
is said to be \grasA{specified by} $g$. 
The role played by $g$ for $Z$ is therefore analogous to a transition 
kernel for a discrete time Markov process, 
except that it allows dependencies on the whole past of the
process.\\

We will always assume that $g$ is \grasA{regular},
which means that it satisfies the following two conditions.
\begin{enumerate}
 \item It is uniformly bounded away from $0$ and $1$: there exists
$\eta>0$ such that
$\eta\leq g(z_0|z)\leq
1-\eta$ for all $z_0\in \bA$, $z\in \bA^{\bN}$.
\item Define the \grasA{variation of $g$ of order $j$} by
\[ \var_j(g)\pordef\sup |g(z_0|z)-g(z_0|z')|\,,\]
where the $\sup$ is over all $z_0\in \bA$, and over all
$z,z'\in \bA^\bN$ for which $z_{i}=z'_{i}$ for all $ 1\leq i \leq j$.
Then $g$ is \grasA{continuous} in the sense that $\var_j(g)\to 0$
when $j\to\infty$.
\end{enumerate}

When $g$ is regular, the existence of at least one stationary process
specified by $g$ follows by a standard compactness
argument (see also the explicit construction given below). 
Once existence is guaranteed, uniqueness can be shown 
under additional assumptions on the speed at which
$\var_j(g)\to 0$. For instance, Doeblin and Fortet
\cite{DoeblinFortet} showed that if
\[\sum_j\var_j(g)<\infty\,,\] 
then there exists a unique process specified by $g$.
More recently, Johansson and \"Oberg \cite{JohanssonOberg} 
strengthned this result, showing that uniqueness holds as soon
as 
\begin{equation}\label{eq:critsuecos}
\sum_j\var_j(g)^2<\infty\,.\end{equation}

An interesting and natural question is to determine if a given
regular $g$-function can lead to a phase transition, that is if it 
specifies at least two \emph{distinct} processes.\\

In a pioneering paper, Bramson and Kalikow \cite{BrKa} 
gave the first example of a regular $g$-function exhibiting a phase
transition. More recently, Berger, Hoffman and Sidoravicius
\cite{BHS}, in a remarkable paper, introduced a
model whose $g$-function also exhibits a phase
transition, but whose variation has a summability that can be made
arbitrarily close to the $\ell^2$-summability of 
the Johansson-\"Oberg criterion (see Remark \ref{rem:variacao} below).\\

The $g$-functions constructed in \cite{BrKa} and \cite{BHS}
have common features. The main one is that they both rely on some 
\emph{majority rule} used in order to fix  
the influence of the past on the probability
distribution of the present. That is, 
$Z_{t+1}$, given $(Z_s)_{s\leq t}$, is
determined by the \emph{sign} (and not the true value) 
of the average of a subset of the variables
$(Z_s)_{s\leq t}$ over a large finite region.
This feature is essential in the mechanisms that lead to
non-uniqueness, since it allows (roughly speaking) 
small local fluctuations to have dramatic effects in the remote
future, thus favorizing the transmission of information from $-\infty$ to
$+\infty$.\\

For the Bramson-Kalikow model, it had already been observed in \cite{Friedli} that
arbitrarily small changes in the behavior of the majority rule,
turning it smooth at the origin,
can have important consequences on uniqueness/non-uniqueness of the
process.\\

In this paper, we give a closer look at a class of models
based on the one of 
Berger, Hoffman, and Sidoravicius (which will be called simply
the BHS-model hereafter). 

Beyond giving a sharper description of the original model of \cite{BHS}, 
our results show that any smoothing of the majority rule leads, under
general assumptions, to uniqueness, even for very slow-decaying variations.

\begin{figure}[H]
\begin{center}
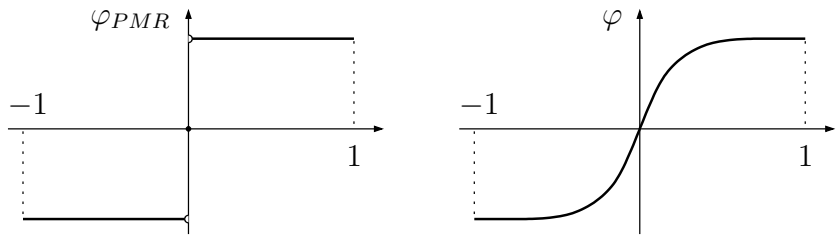
\caption{On the left, the pure majority rule used in \cite{BHS}, for
which non-unicity holds when
$0<\alpha<\frac{1-\epsilon_*}{2}$. On
the right, a smoothed version, for
which the process is always unique for all $\alpha>0$, or more generally, 
for all sequence $h_k\searrow 0$.}
\label{imvarphi}
\end{center}
\end{figure}
We will present these models from scratch, 
and not assume any prior knowledge about \cite{BHS}.
Since their construction is not trivial and deserves some 
explanations, we will state our
results precisely only at the end of Section \ref{sec_BHS}.\\

Before proceeding, we single out other non-uniqueness-related works.
In \cite{Hulse2}, Hulse gave examples of non-uniqueness, based on the Bramson-Kalikow
approach.
In \cite{FM1}, Fern\'andez and Maillard constructed an example, using
a long-range spin system of statistical mechanics, 
although in a non-shift-invariant framework. 
In \cite{GGT}, Gallesco, Gallo and
Takahashi discussed the Bramson-Kalikow model under a different perspective. 

\subsection{Models considered}
Although the basic structure of our model is entirely imported from the one of
BHS, our notations and terminology differ largely from
those of \cite{BHS}.\\

The process $Z=(Z_t)_{t\in \bZ}$ 
defined in \cite{BHS} takes values in an
alphabet with four symbols, where each symbol is actually a pair, which we denote
\[Z_t=(X_t,\omega_t)\,,\]
with~\footnote{Often, we will abbreviate $+1$ (resp. $-1$) by $+$ (resp. $+$).}
$X_t\in \{+,-\}$, $\omega_t\in\{0,1\}$.  The process can be considered as
constructed in two steps. First, a doubly-infinite sequence of i.i.d. 
random variables $\omega=(\omega_t)_{t\in \bZ}$ is sampled, 
representing the \emph{environment}, with distribution $Q$:
\[
Q(\omega_t=1)=1-Q(\omega_t=0)=\tfrac12\,.
\]
Then, for a given environment $\omega$, a process 
$X=(X_t)_{t\in \bZ}$ is considered, whose conditional distribution
given $\omega$ is denoted $P_\omega$ and
called  the \emph{quenched distribution}. We will assume that $P_\omega$-almost surely,
\begin{equation}\label{eq:distr_cond_X}
P_\omega(X_t=\pm |X_{t-1}=x_{t-1},X_{t-2}=x_{t-2},\dots)
=\tfrac12
\bigl\{
1\pm \psi_t^\omega(x_{-\infty}^{t-1})
\bigr\}\,.
\end{equation}
where $x_{-\infty}^{t-1}=(x_{t-1},x_{t-2},\dots)\in \{\pm
\}^\bN$. 
The perturbation $\psi_t^\omega:\{\pm \}^\bN\to [-1,1]$ describes 
how the variables of the process $X$ differ
from those of an i.i.d. symmetric sequence (which corresponds to
$\psi_t^\omega\equiv 0$).
The quenched model will always be \emph{attractive}, in the sense that
$\psi_t^\omega(x_{-\infty}^{t-1})$ is non-decreasing in each of the
variables $x_s$, $s<t$.\\

We assume that the functions $\psi_t^\omega$ satisfy the following conditions:
\begin{itemize}
\item[C1.] For all $x\in \{\pm \}^\bN$,  
$\psi_t^\omega(x)$ depends only
on the environment variables $\omega_s$, with $s$ 
lying at and before time $t$.
\item[C2.] The functions are odd, $\psi_t^\omega(-x)=-\psi_t^\omega(x)$ for all 
$x\in \{\pm\}^\bN$, and bounded 
uniformly in all their arguments:
\[|\psi_t^\omega(x)|\leq \epsilon\quad \text{ for some }\epsilon\in (0,1)\,.\]
\item[C3.] The maps $(x,\omega)\mapsto\psi_t^\omega(x)$ are continuous, 
uniformly in $t$.
\item[C4.] If $\theta:\{0,1\}^\bZ\to\{0,1\}^\bZ$ denotes the shift, 
$(\theta
\omega)_s\pordef\omega_{s+1}$, then 
\[\psi_t^\omega=\psi_0^{\theta_{t}\omega}\,.\]
\end{itemize}

The probability distribution $\bP$ of the
joint process $Z_t=(X_t,\omega_t)$ is defined as follows. 
If $A\in \cF\pordef\sigma(X_t,t\in
\bZ)$, $B\in \cG\pordef \sigma(\omega_t,t\in \bZ)$, then
\begin{equation}\label{Joint_Prob_Measure}
\bP(A\times B)\pordef \int_B P_\omega(A)Q(d\omega)\,.
\end{equation}
We will sometimes denote $\bP$ by $Q\otimes P_\omega$.
It can then be verified that under $\bP$, $Z=(Z_t)_{t\in \bZ}$ is a chain with 
complete connections specified by the regular $g$-function
\begin{align}\label{defg}
g((\pm,\omega_t)|(x_{t-1},\omega_{t-1}),
(x_{t-2},\omega_{t-2}),\dots)&\pordef
\tfrac14\bigl\{1\pm\psi_t^{\omega}(x_{-\infty}^{t-1})\bigr\}\,.
\end{align}

{
Although the processes specified by $g$ are of a dynamical nature (the process
$(x_t,\omega_t)$ at time $t$ having a distribution fixed by the entire past), 
we will rather be working with the quenched picture in mind, and think only 
of the variables
$x_t$ as being dynamical, evolving in a fixed environment $(\omega_t)_{t\in \bZ}$.}\\

The precise definition of the functions
$\psi_t^\omega$ will be given in 
Section \ref{sec:defSpsi}. 
Before that we describe, in an informal way, the 
main ingredients that will appear in their construction.

\subsection{Sampling a random set in the past}
A natural feature of the model is that 
the distribution of the process $X$ at time $t$ 
is determined by its values over a finite (albeit large) region in the
past of $t$. Therefore, for a given environment $\omega$, 
the starting point will be to associate to each
time $t\in \bZ$ a random set $S_t=S_t^\omega$ living in the
{past} of $t$: $S_t\subset (-\infty,t)$.
We will say that $S_t$ {targets} the time $t$. 
Although each $S_t$ is either empty of finite, we will always have, $Q$-almost surely,
\[\sup_t|S_t|=\infty\,\quad\text{ and }\quad 
\sup_t\mathrm{ dist }(t,S_t)=\infty\,.\] 
In the environment $\omega$, 
the distribution of $X_t$ conditionned on its past $(X_s)_{s<t}$ (see
\eqref{eq:distr_cond_X})
is determined by the values of $X$ on $S_t$. 
As a matter of fact, the distribution of $X_t$ will depend on
the \emph{average} of $X$ on the set $S_t$:
\[
\psi_t^\omega(x_{-\infty}^{t-1})=
\text{ odd function of }
\Bigl(
\frac{1}{|S_t^\omega|}\sum_{s\in S_t^\omega}x_s
\Bigr)
\,.
\]
The precise dependence will be fixed by some \emph{majority rule}.

\begin{figure}[H]
\begin{center}
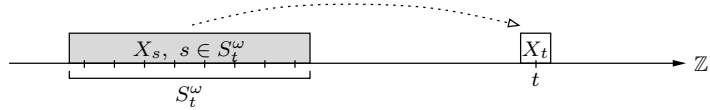
\end{center} 
\caption{In a given environment $\omega$, the distribution of
$X_t$, conditioned on its past, 
is determined by the variables $X_s$, with
$s\in S_t^\omega\subset (-\infty,t)$.
}
\label{fig_setS}
\end{figure}

\begin{rem}
In general, $S_t^\omega$ will \emph{not} be an interval; as will be seen,
$S_t^\omega$ is defined by a multiscale description of $\omega$, linking far apart
intervals in a non-trivial fashion.
Nevertheless, we will simplify the figures by picturing $S_t$ as if it were an interval.
\end{rem}

The sets  $S_t$ will be constructed in such a way that 
the following event occurs with positive $Q$-probability:

\begin{figure}[H]
\begin{center}
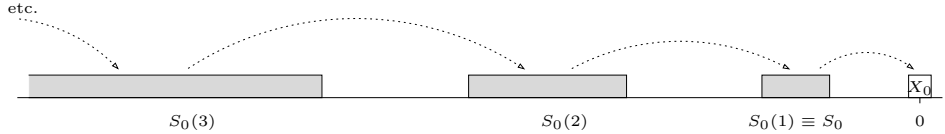
\end{center} 
\caption{An environment in which information is likely to
travel from the remote past up to $0$.}
\label{fig_setSvarios}
\end{figure}

The event depicted represents a global connectivity satisfied
by the sets $(S_t)_{t\in \bZ}$ 
in relation to the origin:
$0$ is targeted by the set $S_0$, which we temporarily denote by $S_0(1)$.
In turn, all points $s\in S_0(1)$ happen to be targeted by the
\emph{same} set, denoted $S_0(2)$. Then, all points $s'\in
S_0(2)$ are targeted by the \emph{same} set
$S_0(3)$, etc.
In this way, for each $j\geq 0$ the variables $\{X_s,s\in S_0(j)\}$, when conditionned on the 
values of the process on the past of $S_0(j)$, are
independent, with a distribution fixed solely by the magnetization
of $X$ on $S_0(j+1)$. In this way, the properties of $X$ in a finite
region of $\bZ$ will be obtained via values of $X$ on a sequence of sets $S_0(j)$,
$j=1,2,\dots$. This sequence will happen to be 
\emph{multiscale} in the sense that $S_0(j+1)$ will
be orders of magnitude larger than $S_0(j)$. Part of the mechanism will be to obtain estimates
on the sizes of these sets. (Obs: The notations of this paragraph will not be used
later. For a precise description of the picture just
described, see the definition of the event $\{\infty\to k\}$ in Section
\ref{subsec:eventoperco}).\\

In general, $|A|$ will denote the number of elements of $A$.
For simplicity, intervals of $\bZ$ will be denoted as
\[\{a,a+1,\dots,b-1,b\}\equiv [a,b]\,.\] 
The \grasA{diameter} of $[a,b]$ is $\diam([a,b])\pordef b-a+1$.
Throughout the paper, 
most objects are random and depend on $\omega$, although
this will not always be indicated in the notations.

\section{The BHS model}\label{sec_BHS}

The construction of the random sets $S_t$ starts by using the 
environment $\omega$ to partition $\bZ$ into blocks of
increasing scales.\\

We start by fixing two numbers:
\[
\epsilon_*\in(0,1)\,,\quad\text{ and }\quad k_*\in \bN\,.
\]
Later, $k_*$ (the smallest scale) will be chosen large. 
For all $k\geq k_*$, define
\[\ell_k\pordef \lceil(1+\epsilon_*)^k\rceil\,,\]
and let $I_k$ be the word defined 
as the concatenation of $\ell_k -1$ symbols 
``$1$'' followed by a symbol ``$0$'': 
\begin{equation}\label{defIk}
I_k=(1,1,\cdots ,1,1,0)\,.
\end{equation}

Let $\omega\in \{0,1\}^\bZ$ be an environment and $[a,b]\subset \bZ$ an
interval of diameter $\ell_k$. 
We say that \grasA{$I_k$ is seen in
$\omega$ on $[a,b]$} if  
\[(\omega_a, \omega_{a+1}, \cdots, \omega_b)=I_k.\]

In a given environment, $I_k$ is seen on infinitely
many disjoint intervals ($Q$-a.s.). Consider 
two successive occurrences of $I_k$ in $\omega$. That is, 
suppose $I_k$ is seen on two disjoint intervals
$[a,b]$ and $[a',b']$, but not on any 
other interval contained in $[a,b']$. Then the 
interval $[b,b'-1]$ is called a \grasA{$k$-block}. By definition, a
$k$-block has diameter at least $\ell_k$, the first
symbol seen on a $k$-block is $0$, and the last $\ell_k-1$ symbols are
$1$s.\\

A given $k$ allows to partition $\bZ$ into
$k$-blocks: for each $t\in \bZ$, there exists a unique
$k$-block containing $t$, denoted by $B^k(t)=[a^k(t),b^k(t)]$, 
where $a^k(t)$ (resp. $b^k(t)$) is the leftmost (resp. rightmost) point of 
$B^k(t)$.

\begin{figure}[H]
 \begin{center}
 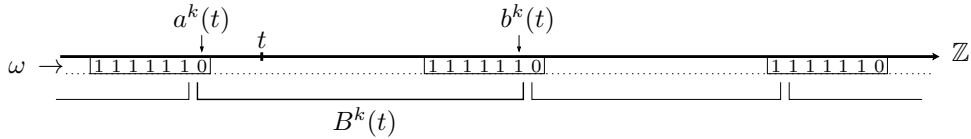
\end{center}
\caption{A  partition of $\bZ$ into $k$-blocks, using
successive occurences of $I_k$ in $\omega$.} 
\label{fig_bloco_simples}
\end{figure}

The diameter of a typical $k$-block is of order
(see Lemma \ref{lematk})
\[\beta_k\pordef 2^{\ell_k}.\] 
In a fixed environment, the partition of $\bZ$ in $k$-blocks is coarser than 
the partition in $(k-1)$-blocks: when $k>k_*$, each 
$k$-block $B$ is a disjoint union of one or more
$(k-1)$-blocks. If we denote the number of $(k-1)$-blocks in
$B$ by $N(B)$, then
\begin{equation}\label{decombloco}
B=b_1\cup b_2\cup \cdots\cup 
b_{N(B)}\equiv\bigcup^{N(B)}_{i=1}b_i,
\end{equation}
where $b_1$ (resp. $b_{N(B)}$) is 
the leftmost (resp. rightmost) $(k-1)$-block contained in $B$.
We will verify in Lemma \ref{lemapik} that $N(B)$ is of order 
\[\nu_k\pordef \frac{\beta_k}{\beta_{k-1}}.\]

When $k>k_*$, the \grasA{beginning of} a $k$-block $B$, decomposed as in
\eqref{decombloco}, is defined as  
\begin{equation}\label{eq_def_begin}
\comeco 
(B)\pordef \bigcup^{N(B)\wedge\lfloor\nu_k^{1-\epsilon_*}\rfloor 
}_{i=1}b_i.
\end{equation}
Due to the exponent ``$1-\epsilon_*$'' in
\eqref{eq_def_begin}, the beginning of a $k$-block, 
when $k>k_*$ is large, is typically smaller than the block
itself (see Lemma \ref{lemma_beginning}):

\begin{center}
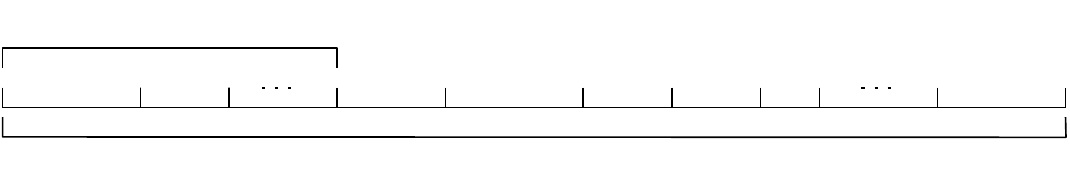
\end{center} 

For a $k_*$-block $B=[a,b]$, the beginning is defined in a different manner:
\[\comeco(B)\pordef \bigl\{s\in B:|s-a|\leq
\beta_{k_*}^{1+\epsilon_*}\bigr\}.\]
Since the typical size of a $k_*$-block $B$ 
is $\beta_{k_*}$, we will verify later that 
$B=\comeco(B)$ with high $Q$-probability. 

\subsection{The definition of $S_t^\omega$ and $\psi_t^\omega$}\label{sec:defSpsi}

In order to help understand the precise definition of $S_t$ given below,
we first give a possible definition which is natural but which is not yet sufficient
for our needs.\\

Fix $t\in \bZ$, and consider the
first scale for which $t$ is not in the beginning of its block:
\[
k_t\pordef \inf \{k\geq k_*:t\not\in
\comeco(B^k(t))\}\,.
\]
Then, a natural way of defining $S_t$ could be
\[
S_t\pordef \comeco(B^{k_t}(t))\,.
\]
Unfortunately, this definition does not guarantee that some event like the one
described after Figure
\ref{fig_setSvarios} occurs with positive probability. Namely, 
two distinct points $t',t''\in S_t$ can very well be targeted by different sets $S_{t'}\neq
S_{t''}$.
The definitions of $S_t$ and $k_t$ thus need to be modified in some subtle way.
\begin{defin}\label{defativos} Let $k\geq k_*$. We say that $t\in \bZ$ is 
\grasA{$k$-active} in the environment 
$\omega$ if for all 
$j\in\{k_*,\cdots, k\}$,
\begin{enumerate}
\item $t \in \comeco(B^j(t))$, where $B^j(t)=[a^j(t),b^j(t)]$ is the
$j$-block containing $t$, and if
\item $|t-a^j(t)|<\beta_{j+1}.$
\end{enumerate}
Let also $\act_{k}\pordef\{t\in \bZ:t\text{ is $k$-active}\}$.
\end{defin}

Observe that
\[\act_{k_*}\supset \act_{k_*+1}\supset\dots\supset
\act_{k}\supset \act_{k+1}\supset\dots\] 
We will see after Lemma \ref{lemma_beginning}
that $\act_{k}\searrow \varnothing$ as $k\to
\infty$, $Q$-almost surely. 
Therefore, it is natural to define, for $t\in \bZ$,
\begin{equation} \label{defkt}
k_t=k_t^\omega \pordef \inf\{k\geq k_*; 
t\notin\act_{k}\}\,,
\end{equation}
with the convention: $\inf \varnothing=\infty$ . 
The set of \grasA{$k$-active points inside a $k$-block} $B$ is
\[\act(B)\pordef\act_k\cap B\,.\]
By definition, $\act(B)\neq \varnothing$, and $\act(B)\subset \comeco(B)$.
Then, let~\footnote{Here, a difference with \cite{BHS}: we don't
require $|S_t|$ to be odd.}
\[S_t =S_t^\omega \pordef \begin{cases} \act(B^{k_t}(t)) & 
\text{if } 
k_*<k_t<\infty\\ \varnothing &\text{otherwise}.\end{cases} \]

By definition, $S_t\subset (-\infty,t)$, and the two
following crucial properties hold:
\begin{enumerate}
\item[P1.]\label{it_lemomega} If $t',t''\in S_t$, then $S_{t'}=S_{t''}$.
\item[P2.]\label{it_lemomega2} If $\omega, \omega'$
are such that $\omega_s=\omega'_s$ for all $s\in (\infty, t]$, 
then $k_t^\omega=k_t^{\omega'}$ and $S_t^\omega=S_t^{\omega'}$.
\end{enumerate}
We can now define $\psi_t^\omega$.

\begin{defin}\label{DefPsi}
Let $\varphi:[-1,1]\to[-1,1]$ be non-decreasing and 
odd, $\varphi(-z)=-\varphi(z)$, and $h_{k}> 0$ be 
a decreasing sequence such that $h_{k}\searrow 0$ when $k\to\infty$. 
If $S_t\neq \varnothing$ and $|t-a^{k_t}(t)|< \beta_{k_t+ 1}$,
define
\begin{equation}\label{defpsi2}
 \psi_t^\omega(x)\pordef
h_{k_t}\varphi\Bigl(\ds
\frac{1}{|S_t|}\sum_{s\in S_t}x_s\Bigr)\,.
\end{equation}
Otherwise,
\begin{equation}\label{defpsi1}
\psi_t^\omega(x)\pordef 0\,.
\end{equation}
\end{defin}

We check that $\psi_t^\omega$ satisfies the properties C$1$-C$4$ described
earlier.
If $h_{k_*}$ is small enough, say $h_{k_*}\leq \tfrac12$, 
then $\psi_t^\omega$ satisfies C$2$.
C$3$ is guaranteed by the fact that $h_{k}\searrow 0$ and that a cutoff was introduced so that 
$\psi_t^\omega=0$ if $|t-a^{k_t}|\geq \beta_{k_t}+1$.
Then,  C$4$ is clearly satisfied, and C$1$ is consequence of P$2$.\\

We will now present some results concerning the processes $Z$
specified by the $g$-function defined in \eqref{defg}, with
$\psi_t^\omega$ defined above.
Our interest will be in observing the role played by
the behavior of $\varphi$ at the origin.

\subsection{A sharper result for the pure majority rule}

In \cite{BHS}, the function $\varphi$ used is a \grasA{pure
majority rule} (see Figure \ref{imvarphi}). 
That is,
\begin{equation}\label{defphibhs}
\varphi_{PMR}(z)\pordef
\begin{cases} 
+1&\text{if } z \in(0,+1]\,,\\ 
0& \text{if } z= 0\,,\\
-1& \text{if } z\in [-1,0)\,.\\
\end{cases}
\end{equation}

The behavior of the model then depends crucially on the choice of the sequence $h_k$.
As will be seen (see \eqref{eq:seriealphaconv} and \eqref{eq:seriealphadiv}), the
criterion is roughly the following: 
\begin{equation}\label{eq:criterium}
\sum_{k}e^{-h_{k+1}^2\beta_k^{1-\epsilon_*}}
\begin{cases}
<\infty & \Rightarrow \text{non-uniqueness}\,,\\
=\infty & \Rightarrow \text{uniqueness}\,.\\
\end{cases}
\end{equation}

The sequence $h_{k}$ considered in \cite{BHS} was therefore of the form
\begin{equation}\label{defhk}
h_{k}\pordef\frac{1}{\beta_{k-1}^{\alpha}}\,,\quad \alpha>0\,.
\end{equation}

With this particular choice, our first result completes the description 
given in \cite{BHS}:

\begin{teo}\label{teo1} Consider the $g$-function 
\eqref{defg}, with $\varphi$ discontinuous at the origin like $\varphi_{PMR}$, and 
$h_{k}$ defined as in \eqref{defhk}.
\begin{enumerate}
\item\label{it1} If $\alpha< \frac{1-\epsilon_*}{2}$, 
then there exist two distinct stationary processes
$Z^+\neq Z^-$ specified by $g$.
\item\label{it2} If $\alpha>\frac{1-\epsilon_*}{2}$, then 
there exists a unique stationary process specified by $g$.
\end{enumerate}
\end{teo}

The methods presented below don't allow to treat the critical case
$\alpha=\frac{1-\epsilon_*}{2}$.
Item \eqref{it1} was the main result of \cite{BHS} but there, the uniqueness 
regime was not studied.

\begin{rem}\label{rem:variacao}
It can be shown (see Appendix \ref{App:variation}) 
that with $h_{k}$ as in \eqref{defhk}, the variation of $g$ satisfies
\[\var_j(g)\leq 
c\cdot j^{-\tfrac{\alpha}{(1+\epsilon_*)^3}}\,,
\]
Therefore, the Johansson-\"Oberg
criterion \eqref{eq:critsuecos} guarantees unicity when 
$\alpha>\frac{(1+\epsilon_*)^3}{2}$.
Our result extends unicity also to values
$\alpha\in(\frac{1-\epsilon_*}{2},\frac{(1+\epsilon_*)^3}{2}]$.
\end{rem}

\subsection{Uniqueness for continuous majority rules}

The following two results show, roughly, that any attempt to turn $\varphi$ smooth at
the origin leads to uniqueness.

\begin{teo}\label{teo2meio} 
Consider the $g$-function \eqref{defg}, with $\varphi$ Lipshitz at the origin:
\begin{equation}\label{eq:conditlimsup}
0\leq  \liminf_{z\to 0^-}\frac{\varphi(z)}{z}
=\limsup_{z\to 0^+}\frac{\varphi(z)}{z}<\infty\,.
\end{equation}
If $h_{k}$ is as in \eqref{defhk}, then for all $\alpha>0$ the
stationary process specified by $g$ is unique.
\end{teo}

Condition \eqref{eq:conditlimsup} 
is satisfied for example when $\varphi$ is
differentiable at $0$: $\varphi'(0)<\infty$.
It will become clear after having read 
the proof that cases where $\varphi$ is continuous at $0$
with $\varphi'(0)=\infty$ can also be handled, but 
uniqueness/non-uniqueness then becomes more sensitive to $\alpha$.
Some examples of non-uniqueness 
with $\varphi'(0)=\infty$ 
for the Bramson-Kalikow model were given in \cite{Friedli}.\\

Under a stronger condition on $\varphi$, we can show that uniqueness
holds for all sequence $h_{k}\searrow 0$.

\begin{teo}\label{teo2} 
Consider the $g$-function 
\eqref{defg}. Assume that $\varphi$ is locally Lipschitz in a neighborhood of the origin: there exists
$\delta>0$ and $\lambda>0$ such that 
\[|\varphi(z_2)-\varphi(z_1)|\leq \lambda |z_2-z_1|\,,\quad \forall z_1,z_2\in[-\delta,\delta]\,.\] 
 Then for any 
sequence $h_{k}\searrow 0$, the stationary process specified by $g$ is unique.
\end{teo}

An example that leads to uniqueness 
for all sequence $h_{k}\searrow 0$ is when $\varphi=\varphi_{\mathrm{lin}}$ 
is linear at the origin: there exists $0<\lambda<\infty$ and 
$\delta>0$ such that $\varphi_{\mathrm{lin}}(z)=\lambda z$ for all
$z\in [-\delta,\delta]$. (This particular example will actually play an important
role in the proof.)
Otherwise, natural candidates such as $\varphi(z)\pordef \tanh(\beta z)$ also
lead to uniqueness even for large $\beta>0$.

\begin{figure}[H]
\begin{center}
\begin{tikzpicture}[scale=2]
\newcommand{\tangh}[1]{(exp(#1)-exp((-1)*#1))/(exp(#1)+exp((-1)*#1))}

\pgfmathsetmacro{\delt}{0.2};
\draw[->](-1.2,0)--(1.2,0);
\draw[->](0,-1.2*0.7)--(0,0.7*1.2) node[left]{$\varphi_{\mathrm{lin}}$};
\draw(-1,0) node[above]{$-1$};
\draw(1,0) node[below]{$+1$};
\draw[dotted] (-1,0)--(-1,-0.7);
\draw[dotted] (1,0)--(1,0.7);
\draw[dashed] ({(-1)*\delt},0)--({(-1)*\delt},-0.7);
\draw[dashed] (\delt,0)--(\delt,0.7);
\draw[very thick] (-1,-0.7)--({-1*\delt},-0.7)--({\delt},0.7)--(1,0.7);
\begin{scope}[xshift=3cm]
\draw[->](-1.2,0)--(1.2,0);
\draw[->](0,-1.2*0.7)--(0,0.7*1.2) node[left]{$\tanh(\beta z)$};
\draw(-1,0) node[above]{$-1$};
\draw(1,0) node[below]{$+1$};
\draw[very thick, domain=-1:1, samples=50] plot (\x,{0.7*\tangh{9*\x}});
\draw[dotted] (-1,0)--(-1,-0.7);
\draw[dotted] (1,0)--(1,0.7);
\end{scope}
\end{tikzpicture}
\end{center}
\end{figure}
{We emphasize that our uniqueness results can't be derived from the 
classical criteria found in the litterature (such as
\eqref{eq:critsuecos}).
The reason for this is that most criteria are insensitive to the behavior of $\varphi$ at
the origin. In particular, our results allow to build $g$-functions with arbitrarily
slow-decaying variation, that specify a unique stationary process.}\\

The paper is organized as follows.
We will first give a detailed description of the environment in 
Section \ref{sec_blocks}, whose spirit follows closely \cite{BHS}.
After that, we will describe when an environment 
should be considered as \emph{good}, and at the
beginning of Section \ref{SecProoft1} introduce an event $\{\infty\to k\}$, that
will be used constantly in the sequel. We then prove Theorems \ref{teo1} and
\ref{teo2meio}, using two propositions that are proved later in
Section \ref{sectionestimprinc}.
Theorem \ref{teo2} is proved in Section \ref{sec:provadiff}.

\section{The environment and properties of blocks}\label{sec_blocks}
In this section, we study typical properties of a $k$-block:
its diameter, its beginning, and finally 
we estimate the number of $k$-active points it contains.\\

Given a $k$-block $B=[a^k,b^k]$ in some environment $\omega$, we define 
$\Pi(B)$, the \grasA{word of} $B$, as the sequence 
of symbols of $\omega$ seen in $B$: 
\[\Pi(B)\pordef (\omega_{a^k},\omega_{a^k+1},\cdots,\omega_{b^k}).\]
Remember (check Figure \ref{fig_bloco_simples}) 
that $\omega_{a^k}=0$, and that
$\omega_{b_k-\ell_k+1}=\dots=\omega_{b_k}=1$.\\

To study properties of blocks that only depend on its word, and since 
the environment is stationnary with respect to $Q$, it
will be enough to consider the blocks containing the origin,
$B^k(0)$, $k\geq k_*$.
The study of $\Pi(B^k(0))$ will be simplified by first
studying the block $B^k(0)$ when its first point is 
fixed at the origin.\\
 
Let therefore  $\eta =(\eta_i)_{i\geq 1}$ be an i.i.d. sequence, 
such that $P(\eta_i=0)=1-P(\eta_i=1)=\tfrac12$. 
For each $k\geq 1$, we consider 
\grasA{the time of first occurrence of $I_k$} (remember
\eqref{defIk}) \grasA{in $\eta$}, defined by
\[\label{deftk} 
T^k\pordef \inf\bigl\{j\geq 
\ell_k:(\eta_{j-\ell_k+1},\cdots,\eta_j)=I_k\bigr\}.\]
Defining $\eta_0\pordef 0$, the random word
\[\Pi^k\pordef 
(\eta_0,\eta_1,\cdots, \eta_{T^k-1})\] 
has the same distribution as of that of a word of a 
$k$-block. We call $\Pi^k$ a 
\grasA{$k$-word}, and denote the set of all $k$-words by 
$\cW^k$.
Many notions introduced for $k$-blocks extend naturally to
the $k$-word $\Pi^k$. For instance, the 
\grasA{diameter} of $\Pi^k$, that is the number of symbols it
contains, is 
\begin{equation}\label{relTk}
\diam(\Pi^k)=T^k\,.
\end{equation}

We will study a few elementary properties of the $k$-word
$\Pi^k$, and then extend them to $B^k(0)$.

\subsection{The diameter}

A classical martingale argument (see ``the monkey typing Shakespeare'' in 
\cite{Williams}) 
allows to compute the expectation of the size of a $k$-word: 
\[E[T^k]=2^1+2^2+\cdots+2^{\ell_k}.\]
Therefore,
\begin{equation}\label{esptk}
 \beta_k \leq E[\diam(\Pi^k)]\leq 
2\beta_{k}\,.
\end{equation}

Moreover, the distribution of $\frac{T^k}{E[T^k]}$ has an 
exponential tail:

\begin{lem}\label{lematk} For all $j\geq 1$
\[ P\bigl(T^k\geq j \beta_k\bigr)\leq e^{-j}.\]
\end{lem}
\begin{proof} For all $0\leq m<j$, let $F_m$ the event in which $I_k$ is 
seen on one of the disjoint subintervals $(m \beta_k,(m+1) \beta_k]$. We 
have
\[P\bigl(T^k\geq j \beta_k\bigr) \leq 
P\Bigl(\bigcap_{m=0}^{j-1}F_m^c\Bigr)=\prod_{m=0}^{j-1}P(F_m^c)=  
P(F_0^c)^j.\]
Moreover, $F_0= \bigcup_{s=0}^{\ell_k-1}G_s$, where 
$G_s\pordef\{I_k$ is seen in some intervals (disjoint)
$s+n\ell_k+[1,\ell_k]$, $n=0,\dots, \beta_k/\ell_k-1\}$.
It can be verified, by inspection, that for each $s$, $P(G_s|G_{s+1}\cup
\dots \cup G_{\ell_k-1})\leq P(G_s)$. This implies 
$P(G_s^c|G_{s+1}^c\cap \dots \cap G_{\ell_k-1}^c)\leq P(G_s^c)$. 
Therefore,
\begin{align*} P(F_0^c) = P\Bigl(\bigcap_{s=0}^{\ell_k-1} G_s^c\Bigr)
&=\prod_{s=0}^{\ell_k-1} P\bigl(G_s^c|G_{s+1}^c\cap \dots 
G_{\ell_k-1}^c\bigr)\\
&\leq 
\prod_{s=0}^{\ell_k-1}P(G_s^c)=\bigl(1-(\tfrac12)^{\ell_k}\bigr)^{ 
\beta_k}\leq e^{-1}.\qedhere
\end{align*}
\end{proof}

\begin{corol}\label{corol_diam} 
For all $j\geq 1$,
\begin{equation}\label{estimdiamBk}
Q\bigl(
\diam(B^k(0))\geq j\beta_k\bigr)\leq 6je^{-j}\,.
\end{equation}
\end{corol}

\begin{proof}
Write $B^k(0)=[a^k(0),b^k(0)]$. For 
all finite interval $J\subset \bN$, 
\begin{align}
Q\bigl(\diam(B^k(0))\in J\bigr)
&=\sum_{\substack{\pi\in  \cW^k\\ \diam(\pi)\in J}}
Q\bigl(\palavra (B^k(0))=\pi\bigr)\nonumber\\
&=\sum_{\substack{\pi\in  \cW^k\\ \diam(\pi)\in J}}
\sum_{-\diam(\pi)<a\leq 0}
Q\bigl(\palavra (B^k(0))=\pi, a^k(0)=a\bigr)\,.\label{eq_decompP}
\end{align}

But $\{\palavra (B^k(0))=\pi, a^k(0)=a\}$, when not empty, is uniquely 
determined by the following three conditions:
 \begin{enumerate}
\item $\omega_j=1$ for all $j\in \{a-\ell_k+1,\dots,a-1\}$,
\item $(\omega_a,\omega_{a+1},\dots,\omega_{a+\diam(\pi)-1})=\pi$,
\item $\omega_{a+\diam(\pi)}=0$.
\end{enumerate}

Therefore, by the independence of the variables $\omega_i$,
 \begin{equation}\label{expr_Q}
Q\bigl(\palavra (B^k(0))=\pi,
a^k(0)=a\bigr)=(\tfrac12)^{\ell_k-1}(\tfrac12)^{\diam(\pi)}(\tfrac12)\,,
\end{equation}
which implies

\begin{align*}
Q\bigl(\diam(B^k(0))\in J\bigr)&=
(\tfrac12)^{\ell_k}\sum_{\substack{\pi\in
\cW^k\\ \diam(\pi)\in
J}}\diam(\pi)(\tfrac12)^{\diam(\pi)}\nonumber\\
&\leq (\tfrac12)^{\ell_k}(\max_{j\in J}j)\sum_{\substack{\pi\in  \cW^k\\
\diam(\pi)\in J}}(\tfrac12)^{\diam(\pi)}.
\end{align*}

But for each $\pi\in \cW^k$, 
\begin{equation} \label{exprPword}P(\Pi^k=\pi)=(\tfrac12)^{\diam(\pi)-1}
\cdot (\tfrac12)=(\tfrac12)^{\diam(\pi)}\,,
\end{equation}
where  ``$\cdot\left(\tfrac12\right)$'' appears in order to 
have a ``$0$'' after $\pi$, to guarantee the occurrence of the event 
$\{\Pi^k=\pi\}$. Therefore,
 \begin{equation}\sum_{\substack{\pi\in  \cW^k\\
\diam(\pi)\in J}}(\tfrac12)^{\diam(\pi)}
\equiv  P\bigl(\diam(\Pi^k)\in J\bigr)=
  P\bigl(T^k\in J\bigr)\,,
\end{equation}
and we have shown that 
\begin{equation}\label{estim_diamblock}
Q\bigl(\diam(B^k(0))\in J\bigr)\leq  (\tfrac12)^{\ell_k}(\max J)
P\bigl(T^k\in J\bigr)\,.
\end{equation}

Let $J^k_i\pordef[i \beta_k,(i+1) \beta_k)$. Since
$(\tfrac12)^{\ell_k}\beta_k=1$,
\begin{align*}
Q\bigl(\diam(B^k(0))\geq j \beta_k\bigr)&=
\sum_{i\geq j}Q\bigl(\diam(B^k(0))\in J^k_i\bigr)\\
&\leq \sum_{i\geq j}(i+1)  P\bigl(T^k\in J_i^k\bigr)\\
&\leq \sum_{i\geq j}(i+1)  P\bigl(T^k\geq i \beta_k\bigr)\,.
\end{align*}
Then, \eqref{estimdiamBk} follows from Lemma \ref{lematk}.
\end{proof}

\subsection{The beginning}
When $k>k_*$, 
$\Pi^k$ can always be viewed as a concatenation of $(k-1)$-words:
\[\Pi^k= 
\Pi^{k-1}_1\circ\cdots\circ\Pi^{k-1}_{N(\Pi^k)}\,,\]
where $\Pi_1^{k-1}\pordef \Pi^{k-1}$ and 
\[
\Pi_j^{k-1}\pordef \Pi^{k-1}\circ \theta_{\diam(\Pi^{k-1}_1\circ \cdots \circ
\Pi^{k-1}_{j-1})}\,.
\]
The number of $(k-1)$-words contained in $\Pi^k$, $N(\Pi^k)$, is geometric:

\begin{lem}\label{lemapik} For all $k>k_*$, 
there exists $p_k\in (0,1)$,
$\nu_k^{-1}\leq p_k\leq 2\nu_k^{-1}$, such that 
\begin{equation}\label{distnpik}
\forall j\geq 1,\quad  
P(N(\Pi^k)=j)=(1-p_k)^{j-1}p_k.
\end{equation}
In particular, $E[N(\Pi^k)]=p_k^{-1}$.
\end{lem}

\begin{proof} Consider an independent identically distributed sequence 
of $(k-1)$-words, with the same distribution as $\Pi^{k-1}$: 
$\Pi^{k-1}_1,\Pi^{k-1}_2,\dots$. When sampling a $(k-1)$-word, we say that this 
$(k-1)$-word is \grasA{closing} if the first occurrence of $I_{k-1}$ 
coincides with the first occurrence of $I_k$, which means that 
the first occurrence of $I_{k-1}$ is preceded by a sequence of 
$\ell_k-\ell_{k-1}$ symbols ``$1$''. Therefore, the concatenation of the 
first $j$ $(k-1)$-words of the sequence $\Pi^{k-1}_1,\Pi^{k-1}_2,\dots$ 
is a $k$-word if and only if the $(j-1)$-th first are not closing and 
the $j$-th is closing. Defining
\[p_k\pordef   P(\Pi^{k-1}\text{ is closing})= P(T^{k-1}=T^k), \] 
we obtain 
\eqref{distnpik}.
If $T^{k-1}=t\geq \ell_k$, we denote by $r_t$ the word seen in the 
interval $[t-\ell_k+1,t-\ell_{k-1}]$, of diameter $\ell_k-\ell_{k-1}$, 
and by $q_t$ the word seen in the interval $[t-\ell_{k-1}+1,t]$.
We have
\begin{align*} P(T^{k-1}=T^k=t)&=P(r_t=(+,\dots,+), T^{k-1}=t)\\
&=P(T^{k-1}>t-\ell_k, r_t=(+,\dots,+), q_t=I_{k-1})\\
&=P(T^{k-1}>t-\ell_k) P(r_t=(+,\dots,+)) P(q_t=I_{k-1})\\
&= \beta_k^{-1}P(T^{k-1}>t-\ell_k)\,.
\end{align*}
Therefore, \[p_k=\sum_{t\geq \ell_k}P(T^{k-1}=T^k=t)= \beta_k^{-1} 
E[T^{k-1}].\]
Using \eqref{esptk}, we get $\nu_k^{-1}\leq p_k\leq 2\nu_k^{-1}$.
\end{proof}

All notions previously defined for blocks, such as active points,
``being good'', etc, have immediate analogs for words. Namely, any $k$-word
$\Pi\in\cP^k$, can be identified as $\Pi(B^k(0))$, where $B^k(0)$ is assumed
to have its first point pinned at the origin: $a^k(0)=0.$  The {beginning
of $\Pi^k$} is the interval
\[
\comeco(\Pi^k) \pordef \Bigl[0, 
\sum^{N(\Pi^k)\wedge\lfloor\nu_k^{1-\epsilon_*}\rfloor}_{i=1} 
\diam(\Pi^{k-1}_i)\Bigr).
\]

Using \eqref{esptk},
\begin{equation}\label{estim_diam_word}
E[\diam(\comeco(\Pi^k))]\leq
\nu_k^{1-\epsilon_*}E[\diam(\Pi^{k-1})]\leq 
2\nu_{k}^{1-\epsilon_*}\beta_{k-1}\,.
\end{equation}

We can now study the position of a point relative to the beginning of each of the $k$-blocks
in which it is contained:

\begin{lem}\label{lemma_beginning}
Let $t\in \bZ$.
For $k=k_*$, 
\begin{equation}\label{estim_comeco_star}
Q\bigl(t \in \comeco(B^{k_*}(t))\bigr)\geq 1-
6\lfloor\beta_{k_*}^{\epsilon_*}\rfloor
e^{-\lfloor\beta_{k_*}^{\epsilon_*}\rfloor}\,.
\end{equation}
For all $k>k_*$,
\begin{equation}\label{estim_comeco}
Q\bigl(t\in \comeco(B^k(t))\bigr)\leq
2\nu_k^{-\epsilon_*}\,,
\end{equation}
\end{lem}

\begin{rem}\label{rem:Kfinito} Observe that $\nu_k$ diverges
superexponentially in $k$, and so \eqref{estim_comeco} implies 
\begin{align*}
\sum_{k>k_*} Q(t \in \comeco(B^{k}(t))<\infty \,.
\end{align*}
Therefore, by the Lemma of Borel-Cantelli, 
$t\not \in \comeco(B^k(t))$ for all large enough $k$. As a consequence,
$Q(k_t<\infty)=1$.
\end{rem}

\begin{proof}[Proof of Lemma \ref{lemma_beginning}:]
It suffices to consider $t=0$.
On the one hand, by Corollary \ref{corol_diam}, 
\begin{align*}
Q\bigl(0\not \in \comeco(B^{k_*}(0))\bigr)&\leq 
Q\bigl(B^{k_*}(0)\setminus \comeco(B^{k_*}(0))\neq\varnothing\bigr)\\
&\leq Q\bigl(\diam(B^k(0))\geq 
\beta_{k_*}^{1+\epsilon_*}\bigr)\leq 
6\lfloor\beta_{k_*}^{\epsilon_*}\rfloor
e^{-\lfloor\beta_{k_*}^{\epsilon_*}\rfloor}\,.
\end{align*} 
On the other hand, using \eqref{expr_Q}, 
\eqref{exprPword}, \eqref{estim_diam_word},
\begin{align*}
Q(0\in\comeco(B^k(0)))
&=\sum_{a\leq 0}Q(0\in\comeco(B^k(0)), a^k(0)=a)\\
&=\sum_{a\leq 0}\sum_{\substack{\pi\in\cW^k:\\
\diam(\comeco(\pi))>|a|}}
Q(\Pi(B^k(0))=\pi, a^k(0)=a)\\
&=(\tfrac12)^{\ell_k}\sum_{\pi\in 
\cW^k}\diam(\comeco(\pi))(\tfrac12)^{\diam(\pi)
}\\
&=\beta_k^{-1} E[\diam(\comeco(\Pi^k))]\leq 2\nu_k^{-\epsilon_*}\,.\qedhere
\end{align*}

\end{proof}

\subsection{The number of active points}

\begin{defin}\label{Def_Bom} 
A $k$-block $B$ is \grasA{good} if
\begin{enumerate}
\item $\diam(B)<\tfrac12 \beta_k^{1+\epsilon_*}$, and if
\item $ n_1(k)\pordef  
\beta_k^{1-\epsilon_*}2^{-k}<\abs{\act(B)}<n_2(k)\pordef  
\beta_k^{1-\epsilon_*} \beta_{k-1}^{2\epsilon_*} $.
\end{enumerate}
If not good, $B$ is \grasA{bad}.
\end{defin}

\begin{prop}\label{lem_bom} 
If $k_*$ if large enough, then
for all $k> k_*$,
\begin{equation}\label{eq_estimblocobom}Q(B^k(0) \text{ is
bad})\leq 2^{-k}.
\end{equation}
\end{prop}

Since the event $\{B^k(0)\text{ is good}\}$ is determined by
the word of $B^k(0)$, we first obtain a similar result for
words. 

The notion of ``good'' extends naturally to $k$-words.
The set of good $k$-words is denoted $\cW^k_{\text{good}}$, and 
$\cW^k_{\text{bad}}\pordef \cW^k\setminus\cW^k_{\text{good}}$.

\begin{lem}\label{Prop_Pkbom} If $k_*$ if large enough, then
for all $k> k_*$,
\begin{equation}\label{eq_induc_k}
P\bigl(\Pi^k\in \cW^k_{\text{bad}}\bigr)\leq3\cdot 
2^{-3k}+2\beta_{k-1}^{-\epsilon_*}\,.
\end{equation}
\end{lem} 

\begin{proof} We write
$\cW^{k}_{\text{good}}
=\cW^{k,1}_{\text{good}}\cap 
\cW^{k,2}_{\text{good}}$, 
where 
\begin{align*}
\cW^{k,1}_{\text{good}}& \pordef
\bigl\{\pi\in 
\cW^k:\,
\diam(\pi)<\tfrac12 \beta_k^{1+\epsilon_*}\mbox{ 
and }\abs{\act(\pi)}> n_1(k)\bigr\}\,,\\
\cW^{k,2}_{\text{good}}& \pordef \bigl\{\pi\in 
\cW^k:\,\abs{\act(\pi)}< n_2(k)\bigr\}\,.
\end{align*}

Let then $\cW^{k,i}_{\text{bad}}\pordef \cW^k\setminus
\cW^{k,i}_{\text{good}}$. 
We first prove that for all $k\geq k_*$,   
\begin{equation}\label{eq_pk1}
P\bigl(\Pi^k\in \cW^{k,1}_{\text{bad}}\bigr)\leq3\cdot 
2^{-3k}.\end{equation}
We will proceed by induction on $k$.
Let $k_*$ be large enough, such that for all 
$k\geq k_*$,
\[e^{-\tfrac{1}{2} \beta_{k}^{\epsilon_*}}
\leq 2^{-3 k}\,,\quad 2\nu_k^{-\epsilon_*}\leq 2^{-3k}\,.\]

(Observe that $k_*\nearrow \infty$ as $\epsilon_*\searrow 0$.)
We start with the case $k=k_*$:
\begin{align*}
P\bigl(\Pi^{k_*}\in \cW^{k_*,1}_{\text{bad}}\bigr)\leq 
P\bigl(\diam(\Pi^{k_*})&\geq \tfrac12 \beta_{k_*}^{1+\epsilon_*}\bigr)\\
+&P\bigl(\diam(\Pi^{k_*})<\tfrac12 \beta_{k_*}^{1+\epsilon_*},
\abs{\act(\Pi^{k_*})}\leq n_1(k_*) \bigr).
\end{align*}
By Lemma \ref{lematk},
\begin{equation}\label{eq_estim1}
P\bigl(\diam(\Pi^{k_*})\geq \tfrac12 \beta_{k_*}^{1+\epsilon_*}\bigr)
=P(T^{k_*}\geq \tfrac12  \beta_{k_*}^{1+\epsilon_*})\leq
e^{-\tfrac{1}{2} \beta_{k_*}^{\epsilon_*}}\leq 2^{-3 k_*}\,.
\end{equation}

On the other hand, $\act(\Pi^{k_*})$ is an interval, and 
$\diam(\Pi^{k_*})<\tfrac12 \beta_{k_*}^{1+\epsilon_*}$
implies that 
$\abs{\act(\Pi^{k_*})}=\diam(\Pi^{k_*})=T^{k_*}$.
Therefore,
\begin{align*}  P\bigl(\diam(\Pi^{k_*})<\tfrac12 
\beta_{k_*}^{1+\epsilon_*},
\abs{\act &(\Pi^{k_*})}\leq n_1(k_*) \bigr)\\
&\leq P\bigl(T^{k_*}\leq  n_1(k_*) \bigr)\\
&\leq n_1(k_*)(\tfrac12)^{\ell_{k_*}}
=\beta_{k_*}^{-\epsilon_*}2^{-k_*}
\leq 2\cdot 2^{-3k_*}\,.
\end{align*}

Therefore, \eqref{eq_pk1} is proved for $k=k_*$. Suppose that 
\eqref{eq_pk1} holds for $k-1$. Remember that $\Pi^k$ is a concatenation 
of $(k-1)$-words, denoted by $\Pi_j$, $j=1,\dots,N(\Pi^k)$. We define 
the events:
\begin{align*}
A_1&\pordef \{d(\Pi^k)<\tfrac12  \beta_k^{1+\epsilon_*}\}\,,\\
A_2&\pordef \{N(\Pi^k)>\nu_k^{1-\epsilon_*}\}\,,\\
A_3&\pordef \Big\{\substack{\text{at least half of 
the $(k-1)$-words}\\\text{ in $\comeco(\Pi^k)$
are in $\cW^{k-1,1}_{\text{good}}$ }}\Big\}\,.
\end{align*}

We claim that $A_1\cap A_2\cap A_3\subset \{\Pi^k\in
\cW^{k,1}_{\text{good}}\}$. Indeed, $A_1$ ensures that the first 
condition in $\cW^{k,1}_{\text{good}}$ is satisfied. 
Furthermore, in 
$A_1\cap A_2\cap A_3$, every $(k-1)$-word $\Pi_j\subset\comeco(\Pi^k)$ 
is at distance $\leq
\tfrac12 \beta_{k}^{1+\epsilon_*}\leq \beta_{k+1}$ of the origin, and 
therefore, each active point of $\Pi_j$ is active in $\Pi^k$. 
As a consequence,

\begin{align*}
\abs{&\act(\Pi^k)}=\sum_{\substack{\Pi_j\subset \comeco(\Pi^k)}}
\abs{\act(\Pi_j)}\\ 
&\geq \sum_{\substack{\Pi_j\subset \comeco(\Pi^k):\\
\Pi_j\in\cW^{k-1,1}_{\text{good}}}}
\abs{\act(\Pi_j)}
> n_1(k-1)\sum_{\substack{\Pi_j\subset \comeco(\Pi^k):\\ 
\Pi_j\in\cW^{k-1,1}_{\text{good}}}}1 \geq n_1(k-1)\cdot \tfrac12 \nu_{k}^{1-\epsilon_*}\equiv
n_1(k)\,.
\end{align*}

Therefore, $\Pi^k\in\cW^{k,1}_{\text{good}}$. 
It follows that 
\[P\big(\Pi^k\in \cW^{k,1}_{\text{bad}}\big)
\leq
P\big(A_1^c\big)+  P\big(A_2^c\big)+  P\big(A_2\cap 
A_3^c\big)\,.\]

As in \eqref{eq_estim1}, 
$P(A_1^c)\leq 2^{-3k}$. 
By Lemma \ref{lemapik},
\begin{align*}
P(A_2^c)&=\sum_{j=1}^{\lfloor\nu_k^{1-\epsilon_*}\rfloor}  P(N(\Pi^k)=j)\\
&=\sum_{j=1}^{\lfloor\nu_k^{1-\epsilon_*}\rfloor}(1-p_k)^{j-1}p_k
\leq \nu_k^{1-\epsilon_*}p_k\leq 2\nu_k^{-\epsilon_*} \leq 2^{-3k}.
\end{align*}
On $A_2$, $\comeco(\Pi^k)$ 
contains exactly $\lfloor\nu_k^{1-\epsilon_*}\rfloor$ $(k-1)$-words. 
Therefore, using the induction hypothesis  \eqref{eq_pk1} for $k-1$,
\begin{align*} 
P(A_2\cap A_3^c)&
=P\Bigl(A_2\cap \Bigl\{\substack{\text{at least half of the ($k-1$)-words 
}\\ \text{of }
\comeco(\Pi^k)\text{ are in }\cW^{k-1,1}_{\text{bad}}} \Bigr\}\Bigr)\\
&\leq  P\Bigl(\Bigl\{\substack{\text{at least half of the 
}\lfloor\nu_{k}^{1-\epsilon_*}\rfloor (k-1)
\text{-words}\\ \text{of }\comeco(\Pi^k) \text{ are in 
}\cW^{k-1,1}_{\text{bad}} }\Bigr\}\Bigr)\\
&\leq
\sum_{j=\tfrac12\lfloor\nu_{k}^{1-\epsilon_*}\rfloor}^{\lfloor\nu_{k}^{
1-\epsilon_*}\rfloor}
\binom{\lfloor\nu_{k}^{1-\epsilon_*}\rfloor}{j}
\bigl(3\cdot 2^{-3(k-1)}\bigr)^j\\
&\leq 
\bigl(2^{-3(k-1)+6}\bigr)^{\tfrac12\lfloor\nu_{k}^{1-\epsilon_*}\rfloor}
<2^{-3k}.      
\end{align*}
This proves \eqref{eq_pk1}.
It remains to prove that for all $k> k_*$,  
\begin{equation}\label{eq_pk2}   
P\big(\Pi^k\in\cW^{k,2}_{\text{bad}}\big)\leq 
2\beta_{k-1}^{-\epsilon_*}\,.
\end{equation}
Since $\act(\Pi^k)\subset \comeco(\Pi^k)$, we estimate
$\diam(\comeco(\Pi^k))$.
We define $T_1^{k-1}:=T^{k-1}$ and for $i>1$,  
\[T_i^{k-1}:=\inf\bigl\{j>T_{i-1}^{k-1}: 
(\eta_{j-\ell_k+1},\cdots,\eta_{j})=I_k
\bigr\}.\]
The increments $\tau_i^{k-1}\pordef T^{k-1}_i-T_{i-1}^{k-1}$ 
are i.i.d., and $E[\tau_1^{k-1}]=E[T^{k-1}]\leq 2\beta_{k-1}$.
Moreover,
$\diam(\comeco(\Pi^k))\leq\sum_{i=1}^{\lfloor\nu_k^{1-\epsilon_*}\rfloor
}\tau_i^{k-1}$. 
Therefore, 
\begin{align*}
P\big(\abs{\act(\Pi^k)}\geq n_2(k)\big)
&\leq   P\Big(\sum_{i=1}^{\lfloor\nu_k^{1-\epsilon_*}\rfloor} 
\tau_{i}^{k-1}\geq  n_2(k)\Big)\leq \frac{E[\tau_1^{k-1}]}{\beta_{k-1}^{1+\epsilon_*}}\leq 
2\beta_{k-1}^{-\epsilon_*}\,.
\end{align*}
This proves \eqref{eq_pk2}. Together, \eqref{eq_pk1} 
and \eqref{eq_pk2} give \eqref{eq_induc_k}.
\end{proof}

\begin{proof}[Proof of Proposition \ref{lem_bom}:] 
Take $k> k_*$, where $k_*$ was defined in the proof of
Lemma \ref{Prop_Pkbom}. We have
\begin{align}
Q(B^k(0) \text{ is bad})\leq 
Q\bigl(\diam(B^k(0))&\geq k \beta_k\bigr)+\nonumber\\
Q\bigl(B^k(0) &\text{ is bad}, \diam(B^k(0))\leq
k \beta_k\bigr)\label{decompruim}\,.
\end{align}
By Corollary \ref{corol_diam}, $Q\bigl(\diam(B^k(0))\geq k
\beta_k\bigr)\leq 6ke^{-k}$.
Then,
\begin{align*}
Q\bigl(B^k(0) \text{ is bad},
\diam(B^k(0))\leq k 
\beta_k\bigr)&=(\tfrac12)^{\ell_k}\sum_{\substack{\pi\in
\cW^k_{\text{bad}}\\ \diam(\pi)\leq k 
\beta_k}}\diam(\pi)(\tfrac12)^{\diam(\pi)}
\\
&\leq k\sum_{\substack{\pi\in
\cW^k_{\text{bad}}}}(\tfrac12)^{\diam(\pi)}
\equiv k   P\bigl(\Pi^k\in  \cW^k_{\text{bad}}\bigr)\,.
\end{align*}

Using Lemma \ref{Prop_Pkbom}, 
\begin{equation}\label{desig2}
Q\bigl(B^k(0) \text{ is bad}\bigr)\leq 6ke^{-k}+ k(2^{-3k}+ 
2k\beta_{k-1}^{-\epsilon_*})\,.
\end{equation}
Taking $k_*$ large enough,
this proves \eqref{eq_estimblocobom}.
\end{proof}

\section{Proofs of Theorems \ref{teo1} and \ref{teo2meio}}\label{SecProoft1}

The proofs of all the results will study the process $X$ under the
quenched measure $P_\omega$, using environments $\omega$ 
for which
the influence of the remote past on the present (for example on a local
event like $\{X_0=+\}$) can be computed
and related to $\varphi$ and to the sequence $h_{k}$.

\subsection{The event $\{\infty\to k\}$}\label{subsec:eventoperco}

To start, consider a set of variables $\{X_s, s\in R\}$, where $R$ is a finite region of
$\bZ$. There clearly exists some $k(1)\geq k_*$ such that $R\subset
B^{k(1)}(0)$. Furthermore, using Remark \ref{rem:Kfinito}, we can
take $k(1)$ sufficiently large,  and guarantee that $R\subset
B^{k(1)}(0)\setminus\comeco(B^{k(1)}(0))$.
But then, by the definition of the $g$-function constructed with $\psi_t^\omega$,
the only way by which the remote past influences the variables in $R$ is through
the value of the 
average of the variables $\{X_t,\,t\in \act(B^{k(1)}(0))\}$.\\

Repeating the same procedure with  $B^{k(1)}(0)$ in place of $R$, 
we deduce that the distribution of $\{X_t,\,t\in \act(B^{k(1)}(0))\}$ is
entirely determined by the values of 
$\{X_t,\,t\in \act(B^{k(2)}(0))\}$ for some sufficiently large $k(2)$,
etc.\\

Our aim will be to make sure that $k(i+1)=k(i)+1$ for all large
$i$, and that the sizes of the sets $\act(B^{k(i)}(0))$ are under control. 
We thus define, for all $k> k_*$,

\begin{equation}\label{eqcalcesp}
\{\infty\to k\}\pordef \bigcap_{j\geq k} \bigl\{
0\not \in \comeco(B^{j}(0)),\,B^{j}(0)\text{ is good}\bigr\}\,.
\end{equation}

The notation used suggests that the event is of the type described earlier in Figure 
\ref{fig_setSvarios}. Indeed, 
let $\omega\in\{\infty\to k\}$. Take $j\geq k$, and $t\in\act(B^j(0))$. 
Since $0\notin \comeco(B^{j+1}(0))$ and
$B^{j+1}(0)=B^{j+1}(t)$, we have that $t\notin \comeco(B^{j+1}(0))$ which
implies $t \notin \act(B^{j+1}(t))$.  Therefore, $k_t=j+1$. 
Moreover, since $B^{j+1}(0)$ is good, we have that $d(B^{j+1}(0))\leq
\tfrac12\beta_{j+1}^{1+\epsilon_*}\leq \beta_{j+2}$. This implies that
$S_t=\act(B^{j+1}(0))$, and 
\begin{equation}\label{eq:psiigual}\psi_t^\omega=h_{j+1}\varphi\Bigl(\frac{1}{|{\act(B^{j+1}(0))}|}
\sum_{s\in \act(B^{j+1}(0))}X_s\Bigr)\,.\end{equation}
Therefore, on $\{\infty\to k\}$, for all $j\geq k$, the 
variables $\{X_t\,,t\in \act(B^j(0))\}$ are i.i.d., and
their distribution is fixed by the value of the magnetization of 
$\{X_t\,,t\in \act(B^{j+1}(0))\}$.
That is, the distribution of the process $X$
on any finite region $X$ is related to the behavior of the non-homogeneous Markov
sequence
\[\xi_{j}\pordef \frac{1}{|\act(B^{j}(0))|}\sum_{s\in\act(B^{j}(0))}
X_s\,,\quad j\geq k\,.\]
The transition probability of the chain will be studied using the following
relation, which holds on the event $\{\infty\to
k\}$, for all $j\geq k$:
\begin{equation}\label{eq:basictrans}
E_\omega[\xi_{j}\dado \xi_{j+1}]=h_{j+1}\varphi(\xi_{j+1})\,,
\end{equation}

\begin{figure}[H]
\begin{center}
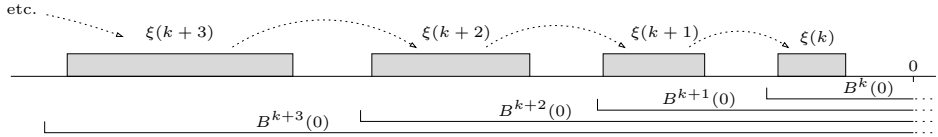
\caption{The quenched distribution of the process $X$, 
on the event $\{\infty\to k\}$.}
\label{Figura_vol_finito}
\end{center}
\end{figure}

\begin{prop}\label{lemaperc} Let $k\geq k_*$.
There exists $\lambda(\epsilon_*,k)>0$ with
$\lambda(\epsilon_*,k)\searrow 0$ when $k\to \infty$, such that
\begin{equation}\label{eq:estim_conn}Q(\infty\rightarrow k)\geq
1-\lambda(\epsilon_*,k)\,.\end{equation}
Moreover, there exists a random scale $K=K(\omega)$, $Q(K<\infty)=1$, such that
\[
Q(\infty \to K)=1\,.
\]
\end{prop}

\begin{proof}
If $k_*$ is as large as in Proposition \ref{lem_bom}, then for $k\geq k_*$
\begin{align}
Q(\{\infty\to k\}^c)\nonumber&\leq \sum_{{j}\geq k}
\bigl\{
Q(B^{j}(0) \text{ is bad})+Q(0\in \comeco(B^{j}(0)))
\bigr\} \nonumber \\
&\leq \sum_{{j}\geq k}
\bigl\{2^{-{j}} +2\nu_{{j}}^{-\epsilon_*}\bigr\}\,, \label{eqkvisivel}
\end{align}
which is summable in $k$. The existence of $K$ follows by the Borel-Cantelli Lemma. 
\end{proof}

\subsection{The measures $P^+_\omega$, $P^-_\omega$ and their 
maximal coupling}\label{Sec_Cons_X}

Our proofs will rely on the use of two particular processes specified by $g$,
$Z^+=(X^+,\omega)$ and  $Z^-=(X^-,\omega)$, symmetric with respect to
each other in the sense that
\begin{equation}\label{eq:simetriaPs}
\bP(X_s^-=+)=\bP(X_s^+=-)=1-\bP(X^+_s=+)\,.
\end{equation}
Each $X^\#$ will actually
be the coordinate process associated to a probability measure
$P_\omega^\#$ on $\{\pm\}^\bZ$ constructed with a pure boundary condition
$\#\in\{+,-\}$. The construction is standard.\\

For $x,y\in\{\pm\}^\bZ$, let $x_s^t\pordef(x_t,\cdots,x_s)$ and
$x_{-\infty}^t\pordef(x_t,x_{t-1},\cdots)$. 
If $\omega\in \{0,1\}^{\bZ}$, define
\[ g^\omega_t\bigl(x_t|x_{-\infty}^{t-1}\bigr)  \pordef 
\tfrac12\{1 +x_t\psi^\omega_t\bigl(x_{-\infty}^{t-1}\bigr)\}\,.\]
Define also the cylinder
$[x]_s^t\pordef \{y\in\{\pm\}^\bZ:y_i =x_i,~i\in [s,t]\}$ and $x_s^ty_{-\infty}^{s-1} 
\pordef (x_t,,\cdots,x_s,y_{s-1},y_{s-2},\cdots)$.
For each $N\in \bN$, $\eta\in \{\pm \}^\bZ$, we 
define a probability measure on $\{\pm\}^{(-N,\infty)}$ by setting
\[P_\omega^{\eta,N}\bigl([x]_{-N+1}^{s}\bigr)\pordef
g^\omega_{-N+1}\bigl(x_{-N+1}
|\eta_{-\infty}^{-N}\bigr)
\prod_{t=-N+2}^{s} g^\omega_t\bigl(x_{t}
|x_{-N+1}^{t-1}\eta_{-\infty}^{-N}\bigr)\,.\]
When $\eta^1\leq \eta^2$ (pointwise),  
$P_\omega^{\eta^1,N}$ and $P_\omega^{\eta^2,N}$ can be coupled as follows.
Consider an i.i.d. sequence of random variables
$(U_t)_{t> -N}$, each with uniform distribution on $[0,1]$.
We construct two processes, $X^1$ and $X^2$, through a 
sequence of pairs, $\Delta_t=\binom{X_t^2}{X_t^1}$, $t>-N$,
in such a way that
$(X^\#_t)_{t>-N}$ has distribution $P_\omega^{\eta^\#,N}$ and such that $X_t^1\leq
X_t^2$ for all $t>-N$.
For $s\leq -N$, set $x^\#_s\pordef \eta_s^\#$.
Assume that the pairs $\Delta_s=\binom{x_s^2}{x_s^1}$ 
have been sampled for all $s<t$, and that these satisfy
$x_s^1\leq x_s^2$. Let 
\begin{equation} \label{eq:defyt}
\Delta_t
=\binom{X_t^2}{X_t^1}\pordef 
\binom{+}{-}1_{A_t}+\binom{+}{+}1_{B_t}+\binom{-}{-}1_{C_t}\,,
\end{equation}
where
\begin{align}
A_t&\pordef \bigl\{0\leq U_t<g_t^\omega(+\dado (x^2)_{-\infty}^{t-1} )-g_t^\omega(+\dado
(x^1)_{-\infty}^{t-1})\bigr\}, \label{eq:defati} \\ 
B_t&\pordef \bigl\{g_t^\omega(+\dado (x^2)_{-\infty}^{t-1} )-g_t^\omega(+\dado
(x^1)_{-\infty}^{t-1})\leq U_t<g_t^\omega(+\dado (x^2)_{-\infty}^{t-1} )\bigr\},\nonumber\\
C_t&\pordef \bigl\{g_t^\omega(+\dado (x^2)_{-\infty}^{t-1}) \leq U_t
\leq 1\bigr\}.\nonumber
\end{align}

We of course have $\mathsf{P}(X_t^1\leq X_t^2)=1$, and 

\begin{align*}
\mathsf{P}\bigl( X_t^2=+ \dado X_{t-1}^2=x^2_{t-1},\cdots,X_{-N+1}^2=x_{-N+1}^2\bigr)
&=\mathsf{P}\bigl(A_t\cup B_t\big)\\
&= g^\omega_t\bigl(+ \dado(x^2)_{-\infty}^{t-1}\bigr)\,,
\end{align*}
and so the distribution of $(X_t^2)_{t>-N}$ is given by $P_\omega^{\eta^2,N}$.
Similarly, 
\begin{align*}
\mathsf{P}\bigl( X_t^1=+ \dado X_{t-1}^1=x^1_{t-1},\cdots,X_{-N+1}^1=x_{-N+1}^1\bigr)
&=\mathsf{P}\bigl(B_t\big)\\
&= g^\omega_t\bigl(+ \dado(x^1)_{-\infty}^{t-1}\bigr)\,,
\end{align*}
and so
the distribution of $(X_t^1)_{t>-N}$ is given by $P_\omega^{\eta^1,N}$.\\

The above coupling allows to extract information about the measures
$P_\omega^{\eta^\#,N}$. 
First, $x^1\leq x^2$ implies $g_t^\omega(+\dado (x^1)_{-\infty}^{t-1} )\leq
g_t^\omega(+\dado (x^2)_{-\infty}^{t-1})$, and so
we always have 
\begin{equation}\nonumber
P_\omega^{\eta^1,N}(X_t=+)\leq P_\omega^{\eta^2,N}(X_t=+)\,,\quad t>-N\,.
\end{equation}
More generally, if $f:\{\pm \}^{(-N+1,\infty)}\to \bR$ is an increasing local function
(that is: non-decreasing in each variable $x_s$ and depending only on a finite number of coordinates), then
\begin{equation}\label{eq:monotaccopl} 
E_\omega^{\eta^1,N}[f]\leq E_\omega^{\eta^2,N}[f]\,.
\end{equation}
Using the previous item, one can also construct two processes $P_\omega^{+}$ and
$P_\omega^{-}$ by taking monotone limits.
Namely, let $P_\omega^{+,N}$ be constructed as above using the boundary condition
$\eta_s\equiv +$ for all $s$. Using \eqref{eq:monotaccopl}, it is easy to see that for all local 
increasing $f$,
\[E_\omega^{+,N+1}[f]\leq E_\omega^{+,N}[f]\,.\]
which allows to define $E^{+}_\omega[f]\pordef \lim_{N\to \infty} E^{+,N}_\omega[f]$.
Since this extends to all continuous function, it defines a measure $P_\omega^+$.
It can then be verified that 
the coordinate process $X=(X_t)_{t\in\bZ}$ defined by
$ X_t(x)\pordef x_t$ satisfies \eqref{eq:distr_cond_X} (with $P_\omega^+$ in place of $P_\omega$).

\subsection{Non-uniqueness when $\alpha<\frac{1-\epsilon_*}{2}$}

Let $\xi^\#$ denote the average of $X^\#$ over $\act(B^k(0))$.
To obtain non-unicity, we will show
that when $\alpha<\frac{1-\epsilon_*}{2}$, 
\begin{equation}\label{eq:desigbasica}
\bP(\xi^+(k)>0)>\tfrac12>\bP(\xi^-(k)>0)\,,
\end{equation}
for some large enough $k$.
Actually, due to the attractiveness of $g$, the following lower bound holds for all
$\omega$:
\begin{equation}\label{cota_simples}
P_\omega^+(\xi_{k}\geq 0)\geq \tfrac12\,.
\end{equation}
\begin{prop}\label{propnonuniq}
Under the hypotheses of Theorem \ref{teo1}, with
$\alpha<\frac{1-\epsilon_*}{2}$, there exists for all
$k'> k_*$ some $\epsilon(k')$,
$\epsilon(k')\searrow 0$ as $k'\to
\infty$, such that 
\begin{equation}\label{Eq_estim_princ}
\forall ~\omega\in \{\infty \to k'\}\,,\quad
P^+_\omega(\xi_{k'}>0)\geq 1-\epsilon(k')\,.
\end{equation}
\end{prop}

\begin{proof}[Proof of Theorem \ref{teo1}, item $(1)$:]
Using \eqref{Eq_estim_princ},
\begin{align*} \bP\bigl(\xi^+(k)>0)&\geq 
\int_{\{\infty \to k\}}P_\omega^+\bigl(\xi_{k}>0\bigr)
Q(d\omega)\\
& \geq \bigl(1-\epsilon(k)\bigr)Q(\infty \to k)\,.
\end{align*}
By taking $k$ large, this lower bound is $>\tfrac12$.
This proves \eqref{eq:desigbasica}, and thereby item \ref{it1} of
Theorem \ref{teo1}.  \end{proof}

\begin{rem}
As the proof of Proposition \ref{propnonuniq} will show, it is possible to
distinguish $X^+$ and $X^-$ even at the origin. Namely it can be shown that
\begin{equation}\label{eq:cotainfestrela}
\forall~
\omega\in \{\infty\to k_*+1\}\,,\quad P^+_\omega(X_0=+)\geq  \tfrac12+\tau\,,
\end{equation}
where $\tau>0$
{once $k_*$ is taken large depending on $\alpha$}. Then, by \eqref{cota_simples} and 
\eqref{eq:cotainfestrela},
\begin{align*}
\bP(X_0^+=+)&\geq (\tfrac12+\tau)Q(\infty\to k_*+1)+\tfrac12 Q(\{\infty\to
k_*+1\}^c)\\
&=\tfrac12+\tau Q(\infty\to
k_*+1)\\
&>\tfrac12\,.
\end{align*} 
Nevertheless, we prefer avoiding having $k_*$ depend on $\alpha$.
\end{rem}

\begin{rem}
In \cite{BHS}, non-uniqueness was obtained by
showing that when $\alpha<\frac{1-\epsilon_*}{2}$, any process $\bP$
specified by $g$ must satisfy 
\[
\bP
\Bigl(
\Bigl\{
\lim_{k\to\infty}1_{\{\xi_{k}>0\}}=1
\Bigr\}
\cup
\Bigl\{
\lim_{k\to\infty}1_{\{\xi_{k}<0\}}=1
\Bigr\}
\Bigr)=1\,.
\]
From this, the existence of two distinct processes 
can be deduced, using an argument based on symmetry and ergodic decomposition.
\end{rem}

\subsection{The signature of uniqueness}
As we have seen, the distribution of $X$ on any finite region can be 
studied via the information contained in the sequence $\xi_{k}$.
On the one hand, 
we have seen in \eqref{eq:desigbasica} that non-unicity is observed through some
asymmetry in the distribution of $\xi_{k}$ when $k$ is large. 
Uniqueness, on the other hand will essentially be characterized by showing that the 
variables $\xi_{k}$ are symmetric:
\begin{equation}\label{eq:signature} 
E_\omega[\xi_{k}]=0\,\quad \text{ for all large }k\,.  
\end{equation}

Observe that regardless of the details of $\varphi$,
\begin{equation}\label{eq:xitendeazero}
\lim_{k\to\infty}E_\omega[\xi_{k}]=0\,
\end{equation}
always holds.
Namely, if $k'$ is large enough so that $\{\infty\to k\}$ for all $k\geq k'$, then
\begin{align*} E_\omega[\xi_{k}]&=E_\omega\bigl[E_\omega[\xi_{k}\dado
\xi_{k+1}]\bigr]\\
&=h_{k+1}E_\omega[\varphi(\xi_{k+1})]=O\bigl(h_{k+1}\bigr)\,.
\end{align*}
More can be said:
when conditioned on $\xi_{k+1}$, $\xi_{k}$ is a Bernoulli sum of i.i.d. variables $X_s$
with expectation $h_{k+1}E_\omega[\varphi(\xi_{k+1})]$.
Therefore, for any fixed $\epsilon>0$, if 
$k$ large enough so that $h_{k+1}\leq \epsilon/2$, a standard large deviation
estimate yields
\begin{equation}\label{eq:estimprobdelta} 
P_\omega\bigl(|\xi_{k}|> \epsilon\dado \xi_{k+1}\bigr)\leq e^{-c|\act(B^k(0))|}\leq
e^{-cn_1(k)}\,,
\end{equation}
where $c=c(\epsilon)>0$ (we have used the fact that $B^k(0)$ is good). Therefore, 
\begin{equation}\label{eq:BorelCantxismall}
\forall \epsilon>0\,,\quad
P_\omega\bigl(
|\xi_{k}|\leq \epsilon\text{ for all large enough }k\bigr)=1\,.
\end{equation}
Therefore, the variables $\xi_{k}$ almost surely tend to zero when $k\to \infty$, and
observing some (a)symmetry in their distribution is a delicate problem.\\

Unicity will be obtained with the help of the following
criterion, whose proof can be found in Appendix \ref{app:unicidade}.

\begin{teo}
Assume that
\begin{equation}\label{eq:uniquenesscrit} 
E_\omega^+[X_t]=0=E_\omega^-[X_t]\,\quad \forall ~t\in \bZ\,.
\end{equation}
Then, $P_\omega^+=P_\omega^-$, and 
any measure $P_\omega$ satisfying
\eqref{eq:distr_cond_X} coincides with $P_\omega^+$ and $P_\omega^-$.
\end{teo}

\subsection{Uniqueness when $\alpha>\frac{1-\epsilon_*}{2}$}

\begin{prop}\label{propespnula} 
Let $\bP=Q\otimes P_\omega$ be the distribution of any process specified
by $g$.
Under the hypotheses of Theorem \ref{teo1}, with
$\alpha>\frac{1-\epsilon_*}{2}$, for $Q$-almost
all environment $\omega$, and for all large enough $k$,
\begin{equation}\label{eq:espxitendeazero} 
\lim_{M\to\infty}E_\omega\bigl[\xi_{k}\Dado \xi_{M}\bigr]= 0\,\quad
P_\omega\text{-almost surely.}
\end{equation}
\end{prop}

\begin{proof}[Proof of Theorem \ref{teo1}, item \ref{it2}:]
By Proposition \ref{lemaperc}, we can consider a fixed environment
$\omega$ for which $K=K(\omega)<\infty$. Take $k\geq K$ large.
We will 
consider $P_\omega^+$, and show that \eqref{eq:espxitendeazero} implies
\eqref{eq:uniquenesscrit}.\\

We know that $\xi_{k}$ is a sum of identically distributed 
variables $X_s$, $s\in \act(B^k(0))$.
By \eqref{eq:espxitendeazero}, $P_\omega^+$-almost surely, for each such $s$,
\begin{align*}\nonumber
\lim_{M\to\infty}P_\omega^+\bigl(X_s=+\Dado \xi_{M}\bigr)
&= \lim_{M\to\infty}\tfrac12\bigl(E_\omega^+\bigl[X_s\Dado \xi_{M}\bigr]+1\bigr)\\
&= \lim_{M\to\infty}\tfrac12\bigl(E_\omega^+\bigl[\xi_{k}\Dado \xi_{M}\bigr]+1\bigr)=
\tfrac12\,.
\end{align*}
This implies that for all large enough $k$, the distribution of $\xi_{k}$ under 
$P_\omega^+(\cdot\dado \xi_{M})$ converges when $M\to\infty$ to a symmetric
distribution.
We can show that this extends to any variable $X_t$ as follows. Take  
$k$ large enough so that $t\not\in \act(B^k(0))$, and write
\begin{equation}\label{eq:machintruc} 
E_\omega^+\bigl[X_t\Dado \xi_{M}\bigr]=
E^+_\omega\bigl[
E^+_\omega[X_t\dado \xi_{k}]\Dado \xi_{M}
\bigr]\,.
\end{equation}
Since $f(x)\pordef E^+_\omega[X_t\dado \xi_{k}=x]$ is odd, $f(-x)=-f(x)$, and since
$\xi_{k}$ converges to a symmetric variable, the right-hand side of 
\eqref{eq:machintruc} converges to zero when $M\to\infty$.
By dominated convergence, we thus get
\[ E_\omega^+[X_t]=\lim_{M\to \infty}E^+_\omega\bigl[
E_\omega^+\bigl[X_t\Dado \xi_{M}\bigr] \bigr]=0\,.
\]
Similarly, $E_\omega^-[X_t]=0$, and this finishes the proof.
\end{proof}

\subsection{Proofs of Propositions
\ref{propnonuniq} and \ref{propespnula}}\label{sectionestimprinc}

The sequence $\xi_{k}\in [-1,1]$ is Markovian and temporally non-homogeneous; we can
nevertheless estimate its transition probabilities with relative precision.
Since the
BHS-model considers the pure majority rule $\varphi_{PMR}$, 
its study can be reduced to the
sign variables 
\[\sigma_{k}\pordef\begin{cases}+1& \text{if } \xi_{k}> 0\,,\\
\phantom{+}0 & \text{if } \xi_{k}=0\,,\\
-1&\text{if }\xi_{k}<0.\end{cases}\]

Remember that if $B^{k}(0)$ is good, then
\begin{equation}\label{tamact} 
n_1(k)< \abs{\act(B^{k}(0))}<n_2(k),
\end{equation}
where the leading term in each $n_\#(k)$ is $\beta_k^{1-\epsilon_*}$. 
But since $h_{k+1}=\beta_k^{-\alpha}$, 
\begin{align}
h_{k+1}^2n_\#(k)
\begin{cases}
\nearrow\infty &\text{ if }\alpha<\frac{1-\epsilon_*}{2}\,,\\
\searrow0 &\text{ if }\alpha>\frac{1-\epsilon_*}{2}\,.
\end{cases}\label{eqconv}
\end{align}

We study the sign changes of the sequence $\xi_{k}$:
\begin{lem}\label{lemaJomega} Assume $\varphi=\varphi_{PMR}$. Let $\omega\in
\{\infty\to k'\}$. There exists $c_0>0$ such that for all large $k\geq k'$, 
\begin{enumerate}
 \item\label{it1lem} For all $h_k\searrow 0$,
\[P_\omega\bigl(\sigma_{k}\leq0\dado\sigma_{k+1}=+\bigr)
 \leq e^{-h_{k+1}^2 n_1(k)/16}\,.
\]
Moreover, if $h_{k+1}^2n_2(k)\nearrow \infty$, 
\[ P_\omega\bigl(\sigma_{k}\leq0\dado\sigma_{k+1}=+\bigr)
\geq c_0 e^{-2h_{k+1}^2 n_2(k)} \,.\]
\item
In particular, if $h_{k}=\beta_{k-1}^{-\alpha}$, 
\[
P_\omega\bigl(\sigma_{k}\leq0\dado\sigma_{k+1}=+\bigr)
\begin{cases}
\leq e^{-h_{k+1}^2 
n_1(k)/16}&\text{ if }\alpha<\tfrac{1-\epsilon_*}{2}\,, \\
\geq c_0/2 &\text{ if }\alpha>\tfrac{1-\epsilon_*}{2}\,.
\end{cases}
\]
\end{enumerate}
\end{lem}

\begin{proof}
On $\omega\in \{\infty\to k'\}$, each block  $B^{k}(0)$, $k\geq k'$, 
is good. In
particular,  \eqref{tamact} holds  and by \eqref{eqcalcesp}, under
$P_\omega\bigl(\,\cdot\,\dado\sigma_{k+1}=+\bigr)$, the
variables $\{X_s, s\in \act(B^{k}(0))\}$ are i.i.d. with
\[
E_\omega[X_s\dado\sigma_{k+1}=+]=h_{k+1}\varphi_{PMR}(\xi_{k+1})\equiv h_{k+1}\,.
\]
Let $X'_s\pordef X_s-E_\omega[X_s\dado\sigma_{k+1}=+]$.
By the Bernstein Inequality,
\begin{align*}
P_\omega(&\sigma_{k}\leq0\dado\sigma_{k+1}=+) \\
& = P_\omega\Bigl(
\tfrac{1}{\abs{\act(B^{k}(0))}}\sum_{s\in \act(B^{k}(0))}
X'_s \leq -h_{k+1} \DDado\sigma_{k+1}=+\Bigr)\\
& \leq e^{-h_{k+1}^2 |\act(B^k(0))|/16} \leq e^{-h_{k+1}^2 n_1(k)/16} .
\end{align*}
To prove the lower bound we let
$A(k)\pordef \abs{\act(B^k(0))}$ and let $\mathcal{L}_k$ denote the set of 
integers between $0$ and $\sqrt{A(k)}$ that have the same parity as $A(k)$.
Using Stirling's formula:
\begin{align*}
 P_\omega\bigl(&\sigma_{k}\leq0\dado\sigma_{k+1}=+\bigr)\\
&\geq \sum_{L\in\cL_k}   P_\omega\Bigl(\sum_{s\in\act(B^k(0))}
X_s=-L\DDado\sigma_{k+1}=+\Bigr)\\
&=\sum_{L\in\cL_k}\binom{A(k)}{\frac{A(k)+L}{2}}
\bigl(\tfrac12(1+h_{k+1})\bigr)^{\frac{A(k)-L}{2}}
\bigl(\tfrac12(1-h_{k+1})\bigr)^{
\frac{A(k)+L}{2}}\\
&\geq
\tilde{c_0} 
e^{-h_{k+1}^2 A(k)}
e^{-\big(h_{k+1}\sqrt{A(k)}\big)/4}\,.\qedhere
\end{align*}
\end{proof}

\begin{proof}[Proof of Proposition \ref{propnonuniq}] 
Let $\omega\in \{\infty\to k'\}$.
The probability we are interested in is defined using the
$+$ boundary condition: 
\[
P_\omega^+(\xi_{k'}>0)=\lim_{N\to
\infty}P_\omega^{+,N}(\xi_{k'}>0)\,.
\]
We choose $N$ large, always to
be between two successive sets $\act(B^{M-1}(0))$,
$\act(B^{M}(0))$. A lower bound is obtained by
assuming that all the sign of the boundary condition travels
down to $\xi_{k'}$: 

\begin{figure}[H]
\begin{center}
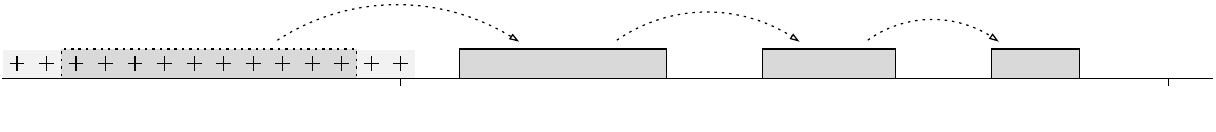
\end{center}
\end{figure}

Using Lemma \ref{lemaJomega},
\begin{align*}
P_\omega^{+,N}(\xi_{k'}>0)
&=P_\omega^{+,N}(\sigma_{k'}=+)\\
&\geq \prod_{k=k'}^{M-1} 
 P_\omega^{+,N}(\sigma_{k}=+\dado\sigma_{k+1}=+)\\
&\geq \prod_{k=k'}^{\infty} 
\{1-e^{-h_{k+1}^2n_1(k)/16}\}\,.
\end{align*}
When $\alpha<\frac{1-\epsilon_*}{2}$, we have 
\begin{equation}\label{eq:seriealphaconv}
\sum_{k}e^{-h_{k+1}^2n_1(k)/16}<\infty\,,
\end{equation}
and so that last product converges and goes to $1$ when $k'\to\infty$.
\end{proof}

\begin{proof}[Proof of Proposition \ref{propespnula}]
By Proposition \ref{lemaperc}, we can consider a fixed
environment $\omega$ for which $K=K(\omega)<\infty$. Then
for each $k'>K$, we have that $\omega\in\{\infty\to
k'\}$.\\

Take $M>k'$. If $\xi_M=0$ then $E_\omega[\xi_{k'}\dado\xi_{M}]=0$.
If $\xi_{M}> 0$, then by the attractiveness of the model, 
$E_\omega[\xi_{k'}\dado\xi_{M}]\geq 0$.
For an upper bound, we look for the scale $k$ at which $\xi$ changes sign:
\begin{equation}\label{eq:deftempoSM}S_M\pordef \max\{k'\leq k\leq M:\xi_{k}\leq 0\},\end{equation}
with the convention that $S_M\pordef k'-1$ if $\xi_{k}>0$ for all $k'\leq k<
M$.
Again by attractiveness, on 
$\{S_M\geq k'\}$, the Markov property gives 
$E_\omega[\xi_{k'}\dado \xi_{S_M}]\leq 0$.
Therefore,  $E_\omega[\xi_{k'}\dado S_M\geq k',\xi_{M}]\leq 0$, and so
\begin{align}
 E_\omega[&\xi_{k'}\dado \xi_{M}]
\leq P_\omega\big(S_M=k'-1\dado \xi_{M}\big)\,.
\end{align}
It follows by Lemma
\ref{lemaJomega} and $1-x\leq e^{-x}$ that 
\begin{align*}
P_\omega\big(S_M=k'-1\dado \xi_{M}\big)
&= P_\omega\big(\xi_{k'}>0,\dots,\xi_{M-1}>0\dado \xi_{M}\big)\\
&=\prod_{k=k'}^{M-1} P_\omega\big(\sigma_{k}=+\dado\sigma_{k+1}=+\big)\\
& \leq \exp\Bigl\{ -c_0\sum_{k=k'}^{M-1}
e^{-2h_{k+1}^2n_2(k)}\Bigr\}.
\end{align*}
But when $\alpha>\frac{1-\epsilon_*}{2}$, 
\begin{equation}\label{eq:seriealphadiv}\sum_{k}
e^{-2h_{k+1}^2n_2(k)}=\infty\,.\qedhere
\end{equation}
\end{proof}

\subsection{Proof of Theorem \ref{teo2meio}}\label{sec:provateo2meio}

The proof of uniqueness when $\varphi$ is Lipshitz at the origin will  be based on the same
principle used when proving item \ref{it2} of Theorem \ref{teo1}.

\begin{prop}\label{propteo2meio}
Let $\bP=Q\otimes P_\omega$ be the distribution of any process specified by
$g$. Under the hypotheses of Theorem \ref{teo2meio}, for $Q$-almost all
environment $\omega$, and for all large enough $k'$, 
\[\lim_{M\to\infty}E_\omega[\xi_{k'}\dado\xi_{M}]=0\quad P_\omega
\text{-almost surely}.\]
\end{prop}

We consider an environment $\omega$ with $K=K(\omega)<\infty$.
We take $k'>K$, and $M>k'$.
As before, the proof is based on showing that whatever the sign of $\xi_{M}$, the
sequence $\xi_{k}$ has a positive probability of having changed sign before reaching
$k'$.\\

We will look at the variables $\xi_{k}$ at even times:
$\xi_{M},\xi_{M-2},\dots$, and show that the probability of $\xi$
changing sign between two scales $k$ and $k-2$ is bounded away from zero.

\begin{lem}\label{lemauniqteo2}
Let $\omega\in\{\infty\to k'\}$ with $k'$ large enough. 
If  $\varphi$ satisfies \eqref{eq:conditlimsup}, 
then for all $\alpha>0$, and all $k\geq k'+2$,
\[P_\omega\bigl(\xi_{k-2}\leq 0\dado\xi_{k}\bigr)\geq 
\tfrac{c_1}{2}>0,\]
where $c_1$ is a universal constant.
\end{lem}

\begin{proof}
If $\xi_{k}\leq 0$, then
attractiveness gives
$P_\omega\bigl(\xi_{k-2}\leq 0\dado\xi_{k}\bigr)\geq \tfrac12$.
If $\xi_{k}=x>0$, let us denote $P_\omega^x(\cdot)\pordef P_\omega(\cdot\dado
\xi_{k}=x)$.
We have
\begin{align*} 
P_\omega^x\bigl( \xi_{k-2}\leq 0\bigr)
&=E_\omega^x\bigl[P_\omega^x\bigl(\xi_{k-2}\leq 0\dado \xi_{k-1}\bigr)\bigr]\\ 
&\geq E_\omega^x\bigl[P_\omega^x\bigl(\xi_{k-2}\leq 0\dado
\xi_{k-1}\bigr)1_{\{|\xi_{k-1}|\leq h_{k-1}^m\}}\bigr]\,,\label{eqxik} 
\end{align*} 
where $m$ is chosen such that 
\begin{equation}\label{doncitalpham}
\alpha>\frac{1-\epsilon_*}{2(1+m)}\,.
\end{equation}
Again, by attractiveness, $-h_{k-1}^m\leq \xi_{k-1}\leq 0$ implies
\[
P_\omega^x\bigl(\xi_{k-2}\leq 0\dado \xi_{k-1}\bigr)\geq \tfrac12\,.
\] 
When $0\leq
\xi_{k-1}\leq h_{k-1}^m$, we will bound this probability using the Central
Limit Theorem.  Let 
\[
\overline{X}_s\pordef
\frac{X_s-E_\omega^x[X_s|\xi_{k-1}]}{\sqrt{\mathrm{Var}_\omega^x(X_s\dado \xi_{k-1})}}\,,
\]
which are centered with variance $1$.
Then
\begin{align*} 
&P_\omega^x\bigl( \xi_{k-2}\leq 0\dado \xi_{k-1}\bigr) \\ 
&= 
P_\omega^x\Bigl(\tfrac{1}{\sqrt{\abs{\act(B^{k-2}(0))}}}\sum_{s\in
\act(B^{k-2}(0))}\overline{X}_s\leq-\tfrac{E_\omega^x[X_s\dado
\xi_{k-1}]\sqrt{\abs{\act(B^{k-2}(0))}}}{\mathrm{Var}_\omega^x(X_s\dado \xi_{k-1})}
\DDado\xi_{k-1}\Bigr)\,.
\end{align*} 

By \eqref{eq:conditlimsup}, there exists $0<\lambda<\infty$ and
$\delta>0$ such that $\varphi(y)\leq \lambda y$ 
for all $0\leq y\leq\delta$. Therefore,
if $k$ is large enough so that $h_{k-1}^m\leq \delta$, 
\begin{align*} 
0\leq E_\omega^x[X_s \dado\xi_{k-1}]
=h_{k-1}\varphi(\xi_{k-1})
\leq \lambda 
h_{k-1}\xi_{k-1}
\leq \lambda h_{k-1}^{1+m}\,.
\end{align*} 
Then, since $B^{k-2}(0)$ is good,
\[\frac{E_\omega^x[X_s\dado
\xi_{k-1}]\sqrt{\abs{\act(B^{k-2}(0))}}}{\mathrm{Var}_\omega^x(X_s\dado \xi_{k-1})}
\leq  
\frac{\lambda h_{k-1}^{1+m}\sqrt{n_2(k-2)}}{\sqrt{1-h_{k-1}^2}}\,.
\] 
But, the dominating term in this last expression is
$\beta_{k-2}^{-(\alpha(1+m)-(1-\epsilon_*)/2)}$, which tends to zero since
\eqref{doncitalpham} holds.
Therefore, taking $k$ large enough,
\begin{align*} 
P_\omega^x\bigl( \xi_{k-2}\leq 0\dado \xi_{k-1}\bigr) &\geq 
P_\omega^x\Bigl(\tfrac{1}{\sqrt{\abs{\act(B^{k-2}(0))}}}\sum_{s\in
\act(B^{k-2}(0))}\overline{X}_s\leq-1\DDado\xi_{k-1}\Bigr) \\
&\geq 
0,5\cdot\tfrac{1}{\sqrt{2\pi}} \int_{-\infty}^{-1}e^{-x^2/2}\,dx\equiv c_1\,.
\end{align*}
It remains to study $P_\omega^x(|\xi_{k-1}|\leq h_{k-1}^m)$. 
We have $E^x_\omega[\xi_{k-1}]=h_{k}\varphi(x)$, so if 
$k$ is such that $h_{k}\leq h_{k-1}^m/2$ then, using the Chebychev Inequality, 
\begin{align*}
P_\omega^x(|\xi_{k-1}|\leq h_{k-1}^m) &\geq P_\omega^x(|\xi_{k-1}-h_{k}\varphi(x)|\leq h_{k-1}^m/2)\\
&\geq 1-\tfrac{2}{\abs{\act(B^{k-1}(0))}h_{k-1}^m}\\
&\geq 1-\tfrac{2}{n_1(k-1)h_{k-1}^m}\,,
\end{align*}
which is $\geq \tfrac12$ when $k$ is large enough.
\end{proof}

\begin{proof}[Proof of Proposition \ref{propteo2meio}]
The proof is the same as the one of Proposition \ref{propespnula}.
If $\xi_{M}>0$, define $S_M$ as in \eqref{eq:deftempoSM}.
Assuming for simplicity that $M-k'$ is even, Lemma \ref{lemauniqteo2} gives
\begin{align*}
 P_\omega\bigl(S_M=k'-1\Dado \xi_{M}\bigr)
&\leq P_\omega\bigl(\xi_{k'}> 0,\xi_{k'+2}>0,\cdots,\xi_{M-2}>0\dado\xi_{M}\bigr)\\
&\leq \bigl(1-\tfrac{c_1}{2}\bigr)^{(M-k')/2}.\qedhere
\end{align*}
\end{proof}
 
\section{Proof of Theorem \ref{teo2}}\label{sec:provadiff}

The proof of uniqueness when $\varphi$ is Lipschitz, for arbitrary
$h_{k}$s, will be based on the same principle used when proving item \ref{it2}
of Theorem \ref{teo1}, showing that for all large enough $k$, the distribution
of $\xi_{k}$ under $P_\omega(\cdot\dado \xi_{M})$ converges to a symmetric
distribution when $M\to\infty$:

\begin{prop}\label{propuniqteo2} Assume $\varphi$ is Lipshitz in a neighborhood of the origin.
Let $h_{k}\searrow 0$ be an arbitrary sequence.
Let $\bP=Q\otimes P_\omega$ be the distribution of any process specified
by $g$.  Then for $Q$-almost
all environment $\omega$, for all large enough $k'$, \[ \lim_{M\to\infty}E_\omega\bigl[\xi_{k'}\Dado
\xi_{M}\bigr]= 0\,\quad P_\omega\text{-almost surely.}\]
\end{prop}

To understand why Lipschitzness near the origin implies uniqueness regardless of
the details of the sequence $h_{k}$, we first consider a particular case.\\

Assume $\varphi$ is globally linear with slope $1$: 
\[
\varphi_{ID}(z)\pordef z\,\quad \forall z\in [-1,1]\,.
\]

Let $\omega\in \{\infty\to k'\}$, and take $k> k'$.
Then
\begin{align*}
E_\omega[\xi_{k}]
=E_\omega\bigl[E_\omega[\xi_{k}\dado \xi_{k+1}]\bigr]
=E_\omega\bigl[h_{k+1}\varphi_{ID}(\xi_{k+1})\bigr]
=h_{k+1}E_\omega[\xi_{k+1}]\,.
\end{align*}
Repeating this procedure we get, for all $L\geq 1$,
\begin{equation}\label{eqiterxi}
E_\omega[\xi_{k}]=\Bigl\{\prod_{j=k+1}^{k+L}h_{j}\Bigr\}E_\omega[\xi_{k+L}]\,.
\end{equation}
Taking $L\to \infty$ gives $E_\omega[\xi_{k}]=0$.\\

The proof of Proposition \ref{propuniqteo2} consists in using this 
phenomenon, which obviously doesn't depend on the precise values of
the sequence $h_{k}$. 
So first, we will consider a case where the Lipschitzness of $\varphi$ is global:

\begin{lem}\label{lemasandro} Assume that $\tvarphi$ is $1$-Lipschitz on 
$[-1,1]$:
\[ |\tvarphi(z_2)-\tvarphi(z_1)|\leq |z_2-z_1|\quad \forall z_1,z_2\in [-1,1]\,.
\]
Let $\bP=Q\otimes P_\omega$ be the distribution of any process specified by the
$g$-function associated to $\tvarphi$ and to some sequence $\thh_{k}\searrow 0$.  
Then for $Q$-almost
all environment $\omega$, for all large enough $k'$, \[
\lim_{M\to\infty}E_\omega\bigl[\txi_{k'}\Dado \txi_{M}\bigr]= 0\,\quad
P_\omega\text{-almost surely.}\] 
\end{lem}

\begin{proof}[Proof of Proposition \ref{propuniqteo2}:]
Let $\delta>0$ and $\lambda>0$ be such that $0\leq \varphi(z_2)-\varphi(z_1)\leq \lambda(z_2-z_1)$ for all
$-\delta\leq z_1\leq z_2\leq \delta$.
We define a function $\tvarphi$ that satisfies the conditions of Lemma
\ref{lemasandro}:
\[ \tvarphi(z)\pordef \tfrac{1}{\lambda}\phi(z)\,,\]
where
\begin{equation}\label{eq:defvarphitil} 
\phi(z)\pordef
\begin{cases} 
\lambda(z+\delta)+\varphi(-\delta) &\text{ if }z\in[-1,-\delta]\,,\\
\varphi(z) &\text{ if }z\in[-\delta,\delta]\,,\\
\lambda(z-\delta)+\varphi(\delta) &\text{ if }z\in[\delta,1]\,.
\end{cases}
\end{equation}

Assume $\omega\in \{\infty\to k'\}$ for some large $k'$.
We fix $M>k'$ large, and to study $E_\omega[\xi_{k'}\dado \xi_{M}]$ we
construct the sequence $\xi_{k}$ for $k$ decreasing from $M$ to $k'$, 
coupled to another sequence $\txi_{k}$; 
$\xi_{k}$ will have its transition probability fixed by $\varphi$ and 
$h_{k}$, 
and $\txi_k$ will have its transition probability
fixed by $\tvarphi$ and
\[ \thh_{k}\pordef \lambda h_{k}\,.
\]
If $\lambda$ is large, we may need to
take $k'$ large enough so that $\thh_{k}\leq 1/2$ for all $k\geq k'$.
The processes $\xi_{k}$ and $\txi_{k}$ will be constructed along  with a sequence
$\gamma_{k}\in \{0,1\}$: if $\gamma_{k}=1$, then $\xi_{k}$ and $\txi_{k}$ are
still coupled; $\gamma_{k}=0$ means they have already decoupled.\\

For simplicity, we will continue denoting the coupling measure by $P_\omega$.
The construction is illustrated on the figure below.
We start with some fixed $\xi_{M}$. By \eqref{eq:BorelCantxismall}, we
can assume that $M$ is large enough in order to guarantee that $|\xi_{M}|\leq \delta$. Let then
$\txi_{M}\pordef \xi_{M}$, and $\gamma_{M}\pordef 1$. \\

Given $(\gamma_{k+1},\txi_{k+1},\xi_{k+1})$, 
$(\gamma_{k},\txi_{k},\xi_{k})$ is constructed as follows:
\begin{enumerate}
\item Sample $\txi_{k}$ as
an average of variables 
$\{\widetilde{X}_s=\pm,s\in \act(B^{k}(0))\}$, i.i.d. with
\[E_\omega[\widetilde{X}_s\dado \txi_{k+1}]=\thh_{k+1}\tvarphi(\txi_{k+1})\,.\]
\item If $\gamma_{k+1}=1$, set $\xi_{k}\pordef \txi_{k}$ and $\gamma_{k}\pordef
1_{\{|\xi_{k}|\leq \delta\}}$.
\item If $\gamma_{k+1}=0$, set $\gamma_{k}\pordef 0$, and
sample $\xi_{k}$ as an average of variables 
$\{X_s=\pm,s\in \act(B^{k}(0))\}$, i.i.d. with
\[ E_\omega[X_s\dado \xi_{k+1}]=h_{k+1}\varphi(\xi_{k+1})\,.\]
\end{enumerate}
By construction, $(\gamma_{k},\txi_{k},\xi_{k})$ is a Markov chain, and 
$\xi_{k}$ has the proper marginals. Namely, if 
$\gamma_{k+1}=0$ then 
\[
E_\omega[\xi_{k}\dado \xi_{k+1}]=h_{k+1}\varphi(\xi_{k+1})\,,
\]
and if $\gamma_{k+1}=1$, then
\begin{align*}
E_\omega[\xi_{k}\dado \xi_{k+1}]=
E_\omega[\txi_{k}\dado \txi_{k+1}]
&=\thh_{k+1}\tvarphi(\txi_{k+1})\\
&=\thh_{k+1}\tvarphi(\xi_{k+1})\\
&=h_{k+1}\varphi(\xi_{k+1})\,,
\end{align*}
where in the last line we used that $|\xi_{k+1}|\leq\delta$.

\begin{figure}[H]
\begin{center}
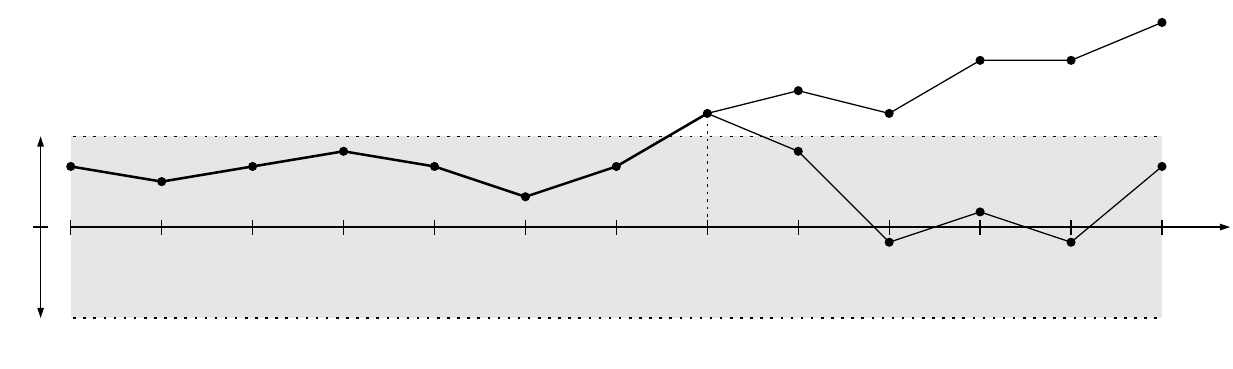
\end{center}
\end{figure}
Let $D$ be the scale at which decoupling occurs: 
\[D\pordef \max\bigl\{k:M> k\geq k', \gamma_{k}=0\bigr\}\,.
\]
Clearly, $D$ is a stopping time for the chain $(\gamma_{k},\txi_{k},\xi_{k})$, and 
if $k\in \{k:M\geq k>D\}$, then 
$\xi_{k}=\txi_{k}$ and $|\xi_{k}|=|\txi_{k}|\leq \delta$.\\

Fix $k'<L<M$.  
On the one hand, proceeding as in \eqref{eq:estimprobdelta},
\begin{align*}
\bigl|E_\omega\bigl[\xi_{k'}, D\geq L\dado \xi_{M}\bigr] \bigr|
&\leq P_\omega(D\geq L\dado \xi_{M})\\
&\leq \sum_{k\geq L}P_\omega(|\xi_{k}|\geq \delta\dado \xi_{M})\leq  \sum_{k \geq
L}e^{-c n_1(k)}\,. 
\end{align*}
On the other hand, $D<L$ implies $\xi_{L}=\txi_{L}$ and so
\begin{align*}
E_\omega\bigl[\xi_{k'}, D< L\dado \xi_{M}\bigr]&=
E_\omega\Bigl[E_\omega\bigl[\xi_{k'}\dado \xi_{L}\bigr]
1_{\{D<L\}}\DDado \xi_{M}\Bigr]\\
&=
E_\omega\Bigl[f_\omega(\txi_{L})
1_{\{D<L\}}\DDado \xi_{M}\Bigr]\\
&=
E_\omega\bigl[f_\omega(\txi_{L})\Dado \xi_{M}\bigr]
+O\bigl(P_\omega(D\geq L\dado \xi_{M})\bigr)\,,
\end{align*}
where $f_\omega(x)\pordef E_\omega[\xi_{k'}\dado \xi_{L}=x]$. 
Now, the construction of $\txi$ was based on $\tvarphi$, so
by Lemma \ref{lemasandro}, 
\[E_\omega[\txi_{L}\dado\txi_{M}]\to 0\quad \text{ when }M\to\infty\,.\]

Since $\txi_{M}\pordef \xi_{M}$, $\txi_{L}$ has, under $P_\omega(\cdot \dado \xi_{M})$ 
in the limit $M\to \infty$, a symmetric law. 
But since  $f_\omega(-x)=-f_\omega(x)$, this implies 
\[ 
E_\omega\bigl[f_\omega(\txi_{L})\Dado \xi_{M}\bigr]
\to 0\quad \text{ as }M\to \infty\,.
\]
We have thus shown that for all $L>k'$,
\[ \limsup_{M\to \infty}\bigl|{E_\omega[\xi_{k'}\dado \xi_{M}]}\bigr|\leq 2\sum_{k\geq
L}e^{-cn_1(k)}\,.\qedhere
\]
\end{proof}

\begin{proof}[Proof of Lemma \ref{lemasandro}] 
{The coupling used below was kindly suggested by S. Gallo.}
We work with two different $g$-functions that have the same sequence $\thh_k$ but different
majority rules.
The first, $\tg$, is associated to
$\tvarphi$, which is $1$-Lipschitz $[-1,1]$.
The second, $\ovg$, is associated to $\ovvarphi\pordef \varphi_{ID}$:
\begin{figure}[H]
\begin{center}
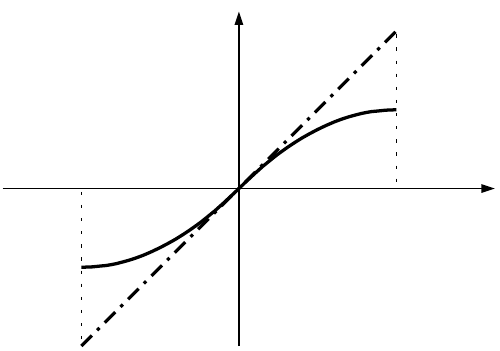
\end{center}
\end{figure}
Uniformly in $z_1<z_2$,
\begin{equation}\label{eq:desigfundam1}
0\leq \tvarphi(z_2)-\tvarphi(z_1)\leq z_2-z_1\equiv \ovvarphi(z_2)-\ovvarphi(z_1)\,.
\end{equation}
Let $\tX$ (resp. $\ovX$) denote the process associated to $\tg$ (resp. $\ovg$). 
Using attractiveness and the notations of Section \ref{Sec_Cons_X},
\begin{align*}
E_\omega^{-,N}[\txi_{k'}]&\leq E_\omega[\txi_{k'}\dado \txi_M]\leq
E_\omega^{+,N}[\txi_{k'}]\,,
\end{align*}
where $N$ is chosen appropriately in function of $M$.
Since $\txi_{k'}$ is an average of identically distributed variables $\tX_s$, our aim will be to show that 
when $M\to\infty$,
\begin{align*}0\leq  E_\omega^{+,N}[\txi_{k'}]-E_\omega^{-,N}[\txi_{k'}]
=E_\omega^{+,N}[\tX_s]-E_\omega^{-,N}[\tX_s]\to 0\,.
\end{align*}
To bound this last difference, 
consider the coupling of $E_\omega^{+,N}$ and $E_\omega^{-,N}$ described in Section \ref{Sec_Cons_X},
which we here denote by $\mathsf{E}_\omega^{\pm,N}$. Since that coupling is maximal,
\begin{equation}\label{eq:usaraccopl} 
E_\omega^{+,N}[\tX_s]-E_\omega^{-,N}[\tX_s]
=2\mathsf{E}_\omega^{\pm,N}[1_{\{\tDelta_s=\binom{+}{-}\}}]\,.
\end{equation}
We will now use \eqref{eq:desigfundam1} to further couple the pair processes,
$\tDelta_s=\binom{\tX^2_s}{\tX^1_s}$ and $\ovDelta_s=\binom{\ovX^2_s}{\ovX^1_s}$. 
This coupling will contain the four processes associated to $\tg$ and $\ovg$, with
boundary conditions $+$ and$-$. 
The coupling will be such that there are more discrepancies between $\ovX^2$ and
$\ovX^1$ than there are between $\tX^2$ and $\tX^1$, in the following sense:
\begin{equation}\label{eq:accoplpares}
1_{\{\tDelta_s=\binom{+}{-}\}}
\leq 
1_{\{\ovDelta_s=\binom{+}{-}\}} \quad a.s.
\end{equation}

By definition, when $s\leq -N$, $\tDelta_s=\ovDelta_s=\binom{+}{-}$.
Let $U_t$, $t>-N$ be an i.i.d. sequence, each with $U_t$ uniform on $[0,1]$.
Assume all pairs, 
$\tDelta_s=\binom{\tx^2}{\tx^1}$ and $\ovDelta_s=\binom{\ovx^2}{\ovx^1}$,
have been sampled for all $s<t$ and that \eqref{eq:accoplpares} holds for all
$s<t$. Let $\tDelta_t$ (resp. $\ovDelta_t$) 
be defined as in \eqref{eq:defyt}, in which $A_t, B_t,C_t$ are replaced by the
corresponding $\tA_t, \tB_t,\tC_t$ (resp. $\ovA_t, \ovB_t,\ovC_t$).
(Obs: we are using the {same} variable $U_t$ to define $\tDelta_t$ and
$\ovDelta_s$.)
Then $\tDelta_t$ and $\ovDelta_t$ obviously have the correct distribution.
To verify that \eqref{eq:accoplpares} holds at time $t$, we first remind that
\begin{align*}
&\ovA_t= \bigl\{0\leq U_t<\ovg_t^\omega(+\dado (\ovx^2)_{-\infty}^{t-1}
)-\ovg_t^\omega(+\dado (\ovx^1)_{-\infty}^{t-1})\bigr\}\,,\\ 
&\tA_t=\bigl\{0\leq U_t<\tg_t^\omega(+\dado (\tx^2)_{-\infty}^{t-1} )-\tg_t^\omega(+\dado
(\tx^1)_{-\infty}^{t-1})\bigr\}\,. 
\end{align*}
Using the fact that $\ovX$ has more discrepancies than $\tX$, and
\eqref{eq:desigfundam1},
\begin{align*}
\ovg_t^\omega(+\dado
(\ovx^2)_{-\infty}^{t-1})-\ovg_t^\omega(+\dado (\ovx^1)_{-\infty}^{t-1})
&=\tfrac{\thh_{k_t}}{2}\Bigl\{
\ovvarphi\Bigl( \tfrac{1}{|S_t|}\sum_{s\in S_t}\ovx^2_s \Bigr)-
\ovvarphi\Bigl( \tfrac{1}{|S_t|}\sum_{s\in S_t}\ovx^1_s \Bigr)
\Bigr\}\\
&\geq \tfrac{\thh_{k_t}}{2}\Bigl\{
\ovvarphi\Bigl( \tfrac{1}{|S_t|}\sum_{s\in S_t}\tx^2_s \Bigr)-
\ovvarphi\Bigl( \tfrac{1}{|S_t|}\sum_{s\in S_t}\tx^1_s \Bigr)
\Bigr\}\\
&\geq\tfrac{\thh_{k_t}}{2}\Bigl\{
\tvarphi\Bigl( \tfrac{1}{|S_t|}\sum_{s\in S_t}\tx^2_s \Bigr)-
\tvarphi\Bigl( \tfrac{1}{|S_t|}\sum_{s\in S_t}\tx^1_s \Bigr)
\Bigr\}\\
&=
\tg_t^\omega(+\dado (\tx^2)_{-\infty}^{t-1})-\tg_t^\omega(+\dado
(\tx^1)_{-\infty}^{t-1})\,,
\end{align*}
which implies $\tA_t\subset \ovA_t$ almost surely.\\

With \eqref{eq:accoplpares} at hand, we go back to \eqref{eq:usaraccopl}:
\begin{align*}
E_\omega^{+,N}[\tX_s]-E_\omega^{-,N}[\tX_s]
&=2\mathsf{E}_\omega^{\pm,N}[1_{\{\tDelta_s=\binom{+}{-}\}}]\\
&\leq 2\mathsf{E}_\omega^{\pm,N}[1_{\{\ovDelta_s=\binom{+}{-}\}}]
=E_\omega^{+,N}[\ovX_s]-E_\omega^{-,N}[\ovX_s]\,.
\end{align*}

But since $\ovvarphi$ is purely linear, the explicit computation made at the
beginning of the section can be repeated, giving
\[
E_\omega^{\pm,N}[\ovX_s]
=\Bigl\{
\prod_{k=k'+1}^M \thh_k
\Bigr\}(\pm 1)
\to 0\quad\text{ when }M\to\infty\,.\qedhere
\]

\end{proof}
\section{Concluding remarks}
The analysis of the model was possible due to the Markovian structure of 
the sequence $\xi_{k}$, in particular to the relation (valid on $\{\infty\to k\}$)
\begin{equation}\label{eq_iterbasico}
E_\omega[\xi_{k}]=h_{k+1}E_\omega[\varphi(\xi_{k+1})]\,.
\end{equation}
We will give a simple heuristic argument that might shed some light on the proofs given
above, and on the role played by the continuity of $\varphi$ at the origin.\\

A mean field approximation consists in assuming that 
$\xi_{k+1}$ can be approximated by its mean: 
\[
\xi_{k+1}\simeq
E_\omega[\xi_{k+1}]\,.\]
This allows to transform
$ E_\omega[\varphi(\xi_{k+1})]\simeq \varphi\bigl(E_\omega[\xi_{k+1}]\bigr) $.
This approximation
is correct in exactly one case: when $\varphi$ is purely linear.\\

With the mean field approximation, one can transform \eqref{eq_iterbasico} into a 
deterministic toy model, in which
$\mu_{k}\pordef E_\omega[\xi_{k}]$ is a sequence 
satisfying the relation
\begin{equation}\label{eq:sistdineff} 
\mu_{k}=h_{k+1}\varphi(\mu_{k+1})\,. 
\end{equation}

We thus take some large integer $M$, fix some 
initial condition, $\mu_{M}$, and study the sequence $\mu_{M},\mu_{M-1},\dots,\mu_{k_*}$.
Since $\varphi(0)=0$, $0$ is always a fixed point for the dynamics.
In the case of a purely linear majority rule, $\varphi(z)=\lambda z$, the trajectory of
$\mu_{k+1}\to \mu_k$ is always attracted
towards the origin, independently of the initial condition. For example,  if $\mu_M>0$:

\begin{figure}[H]
\begin{center}
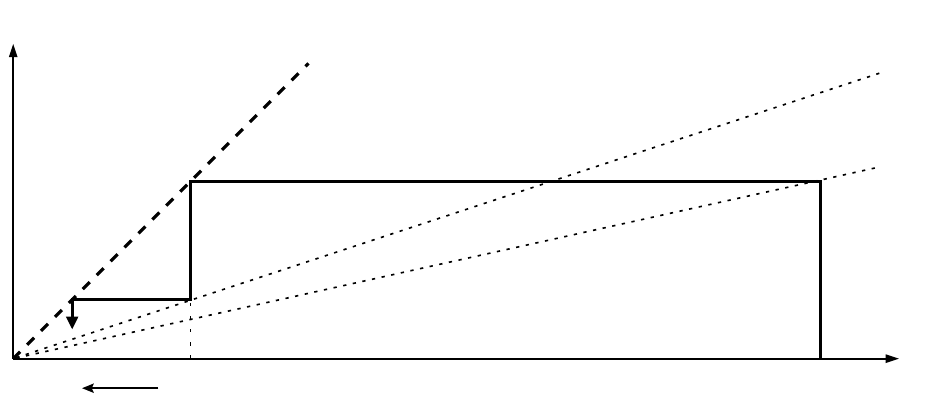
\end{center}
\label{fig:iterreta}
\end{figure}

When $\varphi$ is Lipschitz near the origin, a coupling with a straight line
has shown that the same phenomenon occurs: for large enough $k$,
$z\geq 0$, the curve $z\mapsto h_{k+1}\varphi(z)$  
lies strictly \emph{below} the 
identity $z\mapsto z$, and any initial condition is also attracted towards the origin.\\

In the case of a pure majority rule, the mechanism changes: the trajectories of 
$\mu_{k+1}\to \mu_k=h_{k+1}\varphi_{PMR}(\mu_{k+1})$,
are repelled away from the origin, with a sign that depends on the sign of the initial
condition. For example, if $\mu_M>0$:
\begin{figure}[H]
\begin{center}
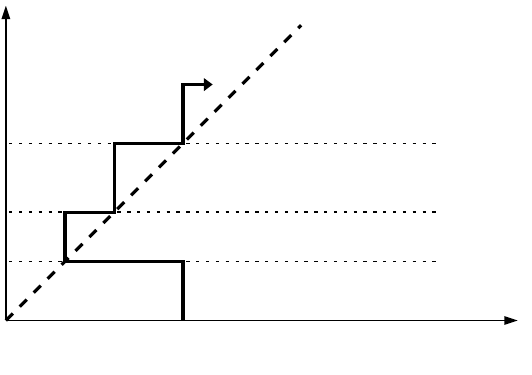
\end{center}
\label{fig:iterrpure}
\end{figure}
This scenario was shown to hold for the BHS model, at least when
$\alpha$ was taken small enough. When $\alpha$ is large, 
fluctuations allow $\mu_k$ to change sign at any time, yielding
unicity.
One can therefore infer that a similar phenomenon can be used to
obtain non-unicity when
$\varphi$ is continuous but with $\varphi'(0)=\infty$.\\

\textbf{Aknowledgements:} We thank Sandro Gallo, Maria Eul\'alia Vares and Vladas
Sidoravicius for various useful discussions. We are in particular grateful to S. Gallo for
suggesting the coupling used in the proof of Lemma \ref{lemasandro}, which lead to a
substantial enhancement of a previous version of Theorem \ref{teo2}.

\appendix

\section{Computing the variation}\label{App:variation}
Consider $(x,\omega)$ and $(\ovx,\ovomega)$ such that
$x_i=\ovx_i$ and $\omega_i=\ovomega_i$
for all $-j\leq i\leq 0$. Let $k$ be such that $\beta_k\leq j<\beta_{k+1}$.
We temporarily denote by $B^m(0,\omega)$ and $B^m(0,\ovomega)$ the $m$-blocks 
associated to the environments $\omega$ and $\ovomega$, respectively. By Definition
\ref{defativos}, the events $\{0\in\act(B^m(0,\omega))\}$ and
$\{0\in\act(B^m(0,\ovomega))\}$ depends only on the interval
$[-\beta_{m+1}-\ell_m,0]$. We will consider two cases.
On the one hand, if $k_0(\omega)\leq k-2$, then
\[\beta_{k_0(\omega)+1}+\ell_{k_0(\omega)}\leq \beta_{k-1}+\ell_{k-2}\leq
\beta_{k}\,,\]
and $[-\beta_{k_0(\omega)+1}-\ell_{k_0(\omega)},0]$ belongs to the interval on which
$(x,\omega)$ and $(\ovx,\ovomega)$ coincide. It implies that
$k_0(\omega)=k_0(\ovomega)$ and that $\act(B^{k_0(\omega)}(0,\omega))=
\act(B^{k_0(\ovomega)}(0,\ovomega))$. Therefore,
\[g\big((x_0,\omega_0)\dado(x,\omega)_{-\infty}^{-1}\big)-g\big((\ovx_0,\ovomega_0
)\dado(\ovx,\ovomega)_{-\infty}^{-1}\big)=0.\]
On the other hand, if $k_{0}(\omega)\geq k-1$, by the previous case, $k_{0}(\ovomega)\geq
k-1$. Since $h_k$ is deacreasing, by Definition \eqref{defg},
\begin{align*}
\abs{g\big((x_0,\omega_0)\dado(x,\omega)_{-\infty}^{-1}\big)
&-g\big((\ovx_0,\ovomega_0)\dado(\ovx,\ovomega)_{-\infty}^{-1}\big)}\\
&\leq \max\{\abs{\psi_0^\omega(x_{-\infty}^{-1})},\abs{\psi_0^\omega(\ovx_{
-\infty}^{-1})}\}\\
& \leq \max\{h_{k_0(\omega)},h_{k_0(\ovomega)}\}\\
&\leq h_{k-1}.
\end{align*}
But if $h_{k-1}=\beta_{k-2}^{-\alpha}$, since $\beta_{k-2}\geq \tfrac12\beta_{k+1}^{(1+\epsilon_*)^{-3}}$,
\begin{align*}
\abs{g\big((x_0,\omega_0)\dado
(x,\omega)_{-\infty}^{-1}\big)-g\big((\ovx_0,\ovomega_0)\dado(\ovx,\ovomega)_{-\infty}^{-1}\big)}
\leq 2^\alpha j^{\tfrac{-\alpha}{(1+\epsilon_*)^3}} \,.
\end{align*}

\section{The uniqueness criterion}\label{app:unicidade}

The uniqueness criterion \eqref{eq:uniquenesscrit} is standard in attractive systems,
although usually used for translation invariant processes (which is not the case of
$P_\omega$).
See for example how it is used in Statistical Mechanics in \cite{Preston}, or in
\cite{Hulse} for $g$-measures.\\

We will use some notations and results from Section \ref{Sec_Cons_X}.
Let $\chi\pordef 
\{\pm \}^{\bZ}=\{x=(x_t)_{t\in\bZ}, x_t=\pm\}$ equipped with the $\sigma$-field
generated by cylinders.
A function $f:\chi\to \bR$ is local if it depends only on a finite
number of $x_t$s; we denote its support by $\supp(f)$.
We say $f$ is increasing if $f(x)\leq f(y)$ whenever $x_t\leq y_t$
for all $t$.\\

The probability measures on $\chi$ are entirely determined by their action on local
functions: if $E_1[f]=E_2[f]$ for all local function $f$, then $P_1=P_2$.
But a local function can always be represented as
\[ f(x)=\sum_{B\subset \supp(f)}\alpha_Bn_B\,,
\]
where $\alpha_B\in \bR$, 
$n_B\pordef \prod_{t\in B}n_t$, with $n_t\pordef \tfrac12(1+x_t)$.

\begin{lem}
If $E_\omega^+[X_t]=E_\omega^-[X_t]$ for all $t\in \bZ$, then
$P_\omega^+=P_\omega^-$.
\end{lem}
\begin{proof}
By what was said above, it suffices to show that $E_\omega^+[n_B]=E_\omega^-[n_B]$ for
all finite $B\subset \bZ$.
Observe that $G=\sum_{t\in B}n_t-n_B$ is increasing. Using
\eqref{eq:monotaccopl} yields 
$E_\omega^+[G]\geq E_\omega^-[G]$, which can be written 
\[ \sum_{t\in B}
\bigl( E_\omega^+[n_t]-E_\omega^-[n_t] \bigr) \geq 
E_\omega^+[n_B]-E_\omega^-[n_B]\geq 0\,.
\]
But $E_\omega^+[n_t]-E_\omega^-[n_t]=E_\omega^+[X_t]-E_\omega^-[X_t]=0$.
\end{proof}

\begin{lem}
If $P_\omega^+=P_\omega^-$, then any other measure $P_\omega$ satisfying
\eqref{eq:distr_cond_X} coincides with $P_\omega^+$ and $P_\omega^-$.
\end{lem}

\begin{proof}
Once again, we need only consider local functions of the form $n_B$. If
$B\subset [a,b]$, we can take $k$ large enough so
that $\act(B^k(0))\subset (-\infty,a\wedge 0)$.
Then, 
\[ 
 E_\omega[n_B]=E_\omega\bigl[E_\omega[n_B\dado \xi_{k}]\bigr]\,.
\]
But, since $n_B$ is increasing,
\begin{align*}
E_\omega[n_B\dado \xi_{k}]  \leq E_\omega[n_B\dado \xi_{k}=+1]
\equiv E_\omega^{+,N(k)}[n_B]\,,
\end{align*}
for some suitable $N(k)$. But by the definition of $P_\omega^+$,
\[\lim_{k\to\infty}E_\omega^{+,N(k)}[n_B]=E^+_\omega[n_B]\,.
\]
We therefore have $E_\omega[n_B]\leq E_\omega^{+}[n_B]$.
In the same, way, $E_\omega^-[n_B]\leq E_\omega[n_B]$.
As a consequence, $E_\omega[n_B]=E_\omega^+[n_B]=E_\omega^-[n_B]$, and the same
extends to all local function $f$.
\end{proof}

\bibliographystyle{plain}
\bibliography{DFbib}

\end{document}